\documentclass{article}

\usepackage{amsthm,amsmath,amssymb,xcolor,stmaryrd,nicefrac,mathtools,multicol,palatino}
\usepackage[inline]{enumitem}

\newtheorem{theorem}{Theorem}[section]

\newtheorem{definition}[theorem]{Definition}
\newtheorem{lemma}[theorem]{Lemma}
\newtheorem{remark}[theorem]{Remark}
\newtheorem{prop}[theorem]{Proposition}

\newtheorem{example}[theorem]{Example}
\newtheorem{corollary}[theorem]{Corollary}
\newtheorem{proposition}[theorem]{Proposition}

\DeclareFontFamily{U}{mathx}{\hyphenchar\font45}
\DeclareFontShape{U}{mathx}{m}{n}{
      <5> <6> <7> <8> <9> <10>
      <10.95> <12> <14.4> <17.28> <20.74> <24.88>
      mathx10
      }{}
\DeclareSymbolFont{mathx}{U}{mathx}{m}{n}
\DeclareFontSubstitution{U}{mathx}{m}{n}
\DeclareMathAccent{\widecheck}{0}{mathx}{"71}
\DeclareMathAccent{\wideparen}{0}{mathx}{"75}

\newcommand{\A}{\mathbb A}
\newcommand{\B}{\mathbb B}

\newcommand{\predek}[2]{\lfloor #1\rfloor_{#2}}

\newcommand{\penum}[1]{\dot #1}

\newcommand{\fs}[2]{[#1 | #2]}
\newcommand{\cfs}[2]{\{ #1 | #2\}}

\newcommand{\fsi}[2]{\cfs {#1}{#2}^* }

\newcommand{\good}[2]{{\G}_{#2}#1}
\newcommand{\G}{\mathcal G}

\newcommand{\begincases}{\begin{enumerate}[label*={\sc Case \arabic*},wide, labelwidth=!, labelindent=0pt]}
\newcommand{\beginclaims}{\begin{enumerate}[label*={\sc Claim },wide, labelwidth=!, labelindent=0pt]}
\newcommand{\begincasesa}{\begin{enumerate}[label={\sc Case ({\rm \roman*})},wide, labelwidth=!, labelindent=0pt]}
\newcommand{\beginsubcases}{\begin{enumerate}[label*= {\rm .\arabic*},wide, labelwidth=!, labelindent=0pt]}

\newcommand{\coeffs}[1]{{\rm C}(#1)}

\newcommand{\coeffsk}[2]{{\rm C}_{#1}(#2)}
\newcommand{\upsquiggly}{\rotatebox[origin=c]{90}{$\leadsto$}}

\newcommand{\chg}[1]{{\uparrow}{#1}}

\newcommand{\val}[1]{|#1|}

\newcommand{\mc}[1]{{\rm mc}(#1)}

\newcommand{\mck}[2]{{\rm mc}_{#1}(#2)}
\newcommand{\bch}[2]{\mathop{\uparrow #2}  #1 }

\newcommand{\bases}[2]{ ^{#2}_{#1}}
\newcommand{\david}[1]{}
\newcommand{\putaway}[1]{}

\newcommand{\gknf}{\nfp k }
\newcommand{\nfp}[1]{=_{#1}}

\newcommand{\ve}{\varepsilon}

\renewcommand{\phi}{{\overline{\varphi}}}

\newcommand{\Om}{{\Omega}}

\newcommand{\al}{{\alpha}}

\newcommand{\be}{{\beta}}

\newcommand{\ga}{\gamma}

\newcommand{\de}{\delta}

\newcommand{\N}{\mathbb N}

\newcommand{\la}{\lambda}

\newcommand{\om}{\omega}

\begin{document}

\title{Fast Goodstein Walks}

\author{David Fern\'andez-Duque and Andreas Weiermann}

\maketitle







\begin{abstract}
We define a variant of the Goodstein process based on fast-growing functions and show that it terminates, although this fact is not provable in Kripke-Platek set theory (or other theories of Bachmann-Howard strength).
We moreover show that this Goodstein process is of maximal length, so that any alternative Goodstein process based on the same fast-growing functions will also terminate.
\end{abstract}

\section{Introduction}

A common reaction to Gödel's proof of incompleteness for Peano arithmetic ($\sf PA$)~\cite{Godel1931} is to argue that the unprovable statement he produced is artificially constructed, casting doubt on whether there are `natural' arithmetical statements which are independent of $\sf PA$.
This establishes the challenge of finding independent statements which do not involve direct coding of metamathematical objects, or other elements which may be deemed extraneous to disciplines such as combinatorics or number theory.

Goodstein's classic principle~\cite{Goodsteinb} is perhaps the oldest example of a statement meeting this challenge.
It is a true statement independent of $\sf PA$~\cite{Kirby} whose understanding requires only high-school level mathematics.
The proof itself does use the well-foundedness of $\varepsilon_0$ \cite{Goodstein1944}, the proof-theoretic ordinal of $\sf PA$.
Goodstein's principle is based on {\em hereditary exponential normal forms,} where\-by each natural number is written in base $k$ in the standard way, as is each exponent appearing in the expansion, and so forth.
Thus for example we may write $18 = 2^{2^2} + 2$ in base-$2$ hereditary exponential normal form.

The Goodstein process then proceeds as follows: given a natural number $m_0$, write $m_0$ in base-$2$ hereditary exponential normal form.
Then, compute $m_1$ by replacing every $2$ appearing in the normal form of $m_0$ by $3$, then subtracting one.
We continue inductively in this fashion, defining $m_{k+1}$ by writing $m_k$ in base-$(k+2)$ normal form, replacing $ k+2 $ by $ k+3 $, and subtracting $1$; the operation of replacing every instance of $ k+2 $ by $ k+3 $ is an instance of the {\em base change operation.}
The Goodstein process terminates if $m_\ell = 0$ for some $\ell$.

For example, if $m_0 = 18$, we would have $m_1 = 3^{3^3} + 2-1 = 3^{3^3} + 1$, which already has thirteen digits.
One can easily check that the sequence continues to grow rather quickly at first.
Nevertheless, Goodstein's principle states that this process always terminates in finite time.

One may wonder if the use of normal forms is essential.
For example, we can write $18 = 2^{2+1}+2^2+2^2+2$.
This begs the question: will Goodstein processes terminate no matter how we write natural numbers?
Surprisingly, the answer is {\em yes.}
A {\em Goodstein walk} is a sequence $m_0,m_1,\ldots$ where each $m_{k+1}$ is defined by writing $m_k$ using {\em any} expression formed from $0,+,(k+2)^x$, then changing the base from $k+2$ to $k+3$ and subtracting one.
The authors have recently shown that every Goodstein walk is finite~\cite{FernandezWWalk}, using the fact that hereditary exponential normal forms are {\em base change maximal,} meaning that applying the base change operator to the normal form of $m$ yields the maximal value among all the possible terms for $m$.
This implies that Goodstein processes using such normal forms have maximal length, hence since these processes are finite, so is any other Goodstein walk.

Goodstein's theorem may also be extended by considering notations based on functions other than the exponential. 
A parametrized version of the Ackermann function gives rise to independence results for theories between $\sf PA$ and the second-order arithmetical theory ${\sf ATR}_0$ of arithmetical transfinite recursion \cite{FSGoodstein}.
This process is based on a notion of normal form based on a `sandwiching' procedure, but other natural notions of normal forms can be considered~\cite{FernandezWCiE}.
More generally, Goodstein walks are also naturally defined in terms of the Ackermann function.
The sandwiching normal forms are base-change maximal in this setting, meaning that every Ackermannian Goodstein walk is finite \cite{FernandezWWalk}.

For writing much bigger numbers, one may switch to notation using even faster-growing functions.
The ordinal $\ve_0$ can be used to define very fast-grow\-ing functions by transfinite recursion.
There are several prominent examples of this, e.g.~the Hardy function \cite{Hardy}.
Goodstein principles based on fast-grow\-ing functions \cite{AraiWW} give rise to independence from theories with strength the Bach\-mann-Howard ordinal, including Kripke-Platek set theory ($\sf KP$) (see e.g.~\cite{Pohlers:2009:PTBook}).

Goodstein walks may also be considered in the context of fast-grow\-ing functions.
Our main goal is to show that these walks are finite, a fact independent of $\sf KP$.
We will use a function that grows slightly faster than standard fast-grow\-ing hierarchies and denote it $\A_k(\xi)$, where $k$ is the number of iterations and $\xi<\ve_0$ is an ordinal.
The precise definition will be given in Section \ref{secAFun}.
Note that Peano arithmetic does not prove that these functions are total, but $\sf KP$ does (as do much weaker theories).

Let us give a brief description of $\sf KP$; a more formal treatment can be found in e.g.~\cite{Barwise}.
The theory $\sf KP$ (with infinity) is axiomatized by all axioms of $\sf ZFC$ except for powerset, but with separation restricted to $\Delta_0$ formulas and replacement restricted to $\Delta_0$-collection.
This theory, while much more powerful than Peano arithmetic, does not prove that the real line exists as a set.
It proves the same arithmetical formulas as other well-known theories such as the theory ${\sf ID}_1$ of non-iterated inductive definitions, and ${\Pi}^1_1$-${\sf CA}^-_0$ of parameter-free ${\Pi}^1_1$ comprehension.
Precise details of these theories are not needed, except that their proof-theoretic ordinal is $\psi(\ve_{\Om+1})$, where $\Om$ is the first uncountable ordinal and $\psi\colon \ve_{\Om+1} \to \Om$ is a function that transforms possibly uncountable ordinals into countable ones.
Most relevant to us is that this ordinal and its system of fundamental sequences can be used to bound the provably total computable functions of any of these three theories.
This allows us to prove that the Goodstein principle based on the $\A_k$ functions is is also unprovable, by showing that the process reaches zero more slowly than stepping down these fundamental sequences.

\section{General outline}

In order to define our Goodstein processes, we must first define the functions $\A_k$ on which they are based.
These functions map ordinals below $\ve_0$ to natural numbers and are defined in terms of {\em fundamental sequences;} each $\xi<\ve_0$ is assigned a sequence $\left ( \fs \xi n\right)_{n<\om}$ such that $\xi=\lim_{n\to\infty} \fs\xi n$ whenever $\xi$ is a limit; for example, one may set $\fs\om n= n$ for all $n$.\footnote{One typically write $\xi[n]$ rather than $\fs\xi n$, but the latter notation will be convenient to avoid ambiguity in expressions such as $\al\be[n]$.}
Fundamental sequences are discussed in Section \ref{secOrd}, along with general background in ordinal arithmetic used in the text.

We can use these fundamental sequences to define fast-growing functions by transfinite recursion.
For $k\geq 2$, the function $\A_k\colon\ve_0\to\mathbb N$ is first defined by setting $\A(\xi)=\xi +1$ if $\xi<\om$.
If $\xi =  \al+b  $ with $\al$ a limit, we assume inductively that $\A_k(\xi')$ is defined for all $\xi'<\xi$ and define $\A_k(\xi)$ by $k$-fold iteration along the fundamental sequence of $ \al$ applied to $\A_2( \xi-1)$: for example,
\[
\A_2(\xi) = \A_2 \Big ( \fs{ \al}{\A_2\big (\fs{ \al}{\A_2( \xi-1)}\big ) }\Big).
\]
$\A_3( \xi)$ is defined with an additional nesting, and in general
\[\A_k( \xi)  =  \A_k \Big ( \fs{ \al}{\A_k\big (\ldots \fs{ \al}{\A_k( \xi-1)} \ldots \big ) }\Big),\]
with $k+1$ nested occurrences of $\A_k$.
Note that we do not assume that $b=0$, in which case the notation $\A_k( \xi-1)$ is defined as the {\em maximal coefficient,} $\mc\xi$, of $\xi$; that is, the largest $n$ so that there is a sub-expression of the form $\om^\ga\cdot n$ in the Cantor normal form of $\xi$.
For example, the maximal coefficient of $\om$ is $1$, so that
\[
\A_2(\om) = \A_2 \Big ( \fs{ \om}{\A_2\big (\fs{ \om}{1}\big ) }\Big) = 1+1+1 = 3.
\]
A rigorous definition of these functions is given in Definition \ref{defA}, and their basic properties are established throughout Section \ref{secAFun}.

We can then use the $\A_k$ functions to write natural numbers in various ways; we can write $3=\A_2(\om)$ as above, or simply $3=\A_2\Big(\A_2 \big(\A_2 (0) \big ) \Big)$.
For any such expression, we may change the base from $2$ to $3$ by replacing every subindex $2$ by $3$; more generally, we define $\bch t{\bases k\ell}$ to be the result of replacing every subindex $k$ by some $\ell>k$ in a term $t$ built up from $0$, $\A_k$, and ordinals below $\ve_0$.
We may even use the function $\A_k$ itself in the notations for these ordinals and write e.g.~$\A_2(\om^{\A_2(\om)})$ instead of $\A_2(\om^{3})$.
Note that the value of this term is already quite large, and the value of $\bch {\A_2(\om^{\A_2(\om)}) }{\bases 23} = \A_3(\om^{\A_3(\om)}) $ is much larger.
With these elements, we may define a {\em Goodstein walk} to be any sequence of natural numbers $m_0,m_1,\ldots$ such that $m_{i+1}$ is obtained by choosing a base-$(i+2)$ term $t_i$ with value $m_i$ and letting $m_{i+1}$ be $\bch {t_i}{\bases {i+2}{i+3}} - 1$.

Our first main result is that any Goodstein walk is finite.
We prove this by choosing the terms $t_i$ in a canonical way, called the {\em normal form of $m_i$,} and showing that the Goodstein process based on these normal forms always terminates in finite time.
Moreover, it gives the longest possible termination time over any other choice of terms.
The normal form for $m$ is selected via a sequence $\xi_0,\ldots,\xi_n$ of ordinals below $\ve_0$ such that $\xi_0 = 0$ and $\A_k(\xi_n) = m$, and for each $i<n$ we have that $\xi_{i+1}$ is the maximal ordinal such that $\A_k(\xi_{i+1}) \leq m$ and $\mc{\xi_{i+1}} \geq \A_k(\xi_i)$.
The intuition is that very large ordinals are nested within the coefficients in the notation for $m$, and these lead to a maximal increase after base change, even if $\xi_n$ itself is relatively small as an ordinal.
In Section~\ref{secNF} we define these normal forms rigorously (see Definition \ref{defNF}) and in Section~\ref{secBC} we study the base change operation based on these normal forms.
With this, we show in Section~\ref{secMax} that our normal forms are indeed base-change maximal (Theorem \ref{theoMax}).
In preparation for our termination proof, Section~\ref{secNorm} shows that normal forms are preserved under base-change.
Section~\ref{secWalk} then uses these results to establish that the Goodstein process based on these normal forms are finite (Theorem \ref{theoGood}) and, by maximality, all Goodstein walks based on the $\A_k$ functions are finite (Theorem \ref{theoWalk}).

The proof of termination uses an ordinal assignment, where we define an additional function $\A_\om$ based on the $\vartheta$ function~\cite{BuchholzOrd}.
Let us denote the first uncountable ordinal by $\Om$, so that the next $\ve$-number is $\ve_{\Om+1}$.
The function $\A_\om$ maps $\ve_{\Om+1}$ to $\Om$, providing notations for the Bachmann-Howard ordinal; this function is studied in~\cite{FWTheta}, where it is denoted $\sigma$.
We extend the base-change operation by letting $\bch t{\bases k\om}$ be the result of replacing every occurrence of $\om$ by $\Om$ and every occurrence of $\A_k$ by $\A_\om$.
The ordinals thus assigned to elements of the Goodstein process are decreasing, from which we deduce that the Goodstein process based on normal forms is terminating.
All other Goodstein walks are also terminating by maximality, even if they do not happen to be decreasing with respect to this ordinal assignment.

The rest of the paper is devoted to establishing that this new Goodstein principle is independent of $\sf KP$ (Theorem~\ref{theoGoodInc}); Section~\ref{secThetaFun} reviews fundamental sequences for the Bachmann-Howard ordinal, and Section~\ref{secInd} shows that the Goodstein process terminates more slowly than the process of stepping down these fundamental sequences, from which independence follows by classic results in proof theory.
Section~\ref{secConc} provides some concluding remarks and open questions.

\section{Ordinals below $\ve(\kappa)$}\label{secOrd}


In this section, we review some elementary notions from ordinal arithmetic.
Some notions covered here are only used much later in the text, so the reader may prefer to skim this section and return to it as a reference.
We assume basic familiarity with ordinal addition, multiplication, and exponentiation.
The predecessor of $\al$ will be denoted $\al-1$, when it exists.
We will use the following simple inequalities, proven by routine induction.

\begin{lemma}\label{lemOrdIneq}
Let $\al,\be$ be ordinals.
\begin{enumerate}

\item If $\al,\be> 1$ then $\al+\be\leq \al\be$.

\item If $\al>1$ then $\be\leq \al^\be$.

\end{enumerate}
\end{lemma}

Throughout the text, we will use normal forms for ordinals based on either $\om$ or $\Om$, the first uncountable ordinal.
Let $\kappa\in \{\om,\Om\}$ and $\xi>0$ be an ordinal.
There exist unique ordinals $\al,\be,\ga$ with $\be<\kappa$ such that $ \xi=\kappa^\al\be+\ga$ and $\ga<\kappa^\al$.
This is the {\em $\kappa$-normal form of $\xi$.}
Note that the $\omega$-normal form of $\xi$ is not precisely its Cantor normal form, as Cantor normal forms do not involve coefficients.
For our purposes, we simply say that an expression $\al+\be$ is in {\em Cantor normal form} if $\al+\be>\al'+\be $ for all $\al'<\al$.

The ordinal $\ve(\kappa)$ is defined as the least $\ve>\kappa$ such that $\ve=\om^\ve$.
If $\xi=\kappa^\al\be+\ga <\ve(\kappa)$ is in $\kappa$-normal form, then we may also deduce that $\al<\xi$.
We write $\ve_0=\ve(\om)$ and $\ve_{\Om+1} =\ve(\Om)$ as is standard.
Define $\kappa_0(\alpha) = \alpha$ and $\kappa_{i+1}(\alpha) = \kappa^{\kappa_i(\alpha)}$.
We set $\kappa_i = \kappa_i(1)$.
Then, $\ve(\kappa) = \sup_{n<\om}\kappa_i$.

Next we define fundamental sequences.
We will reserve the more standard notation $\xi[n]$ for arbitrary systems of fundamental sequences, whereas the notation $\fs\xi n_\kappa$ refers exclusively to the operation defined below.

\begin{definition}
Let $\kappa$ be an ordinal and $\ve=\ve(\kappa)$ the least $\ve$-number above $\kappa$.
We define fundamental sequences $\big (\fs \xi\theta_\kappa \big )_{\theta<\ve}$ and $\theta<\kappa$ recursively as follows, where we assume that $\xi$ is written in $\kappa$-normal form.
\begin{enumerate}

\item $\fs 0\theta_\kappa = \fs 1\theta_\kappa = 0$,

\item $ \fs{ \kappa^\alpha\be + \ga }{\theta}_\kappa =  \kappa^\alpha\be + \fs{ \ga }{\theta}_\kappa $ if $\ga>0$,

\item $ \fs{ \kappa^\alpha\be   }{\theta}_\kappa =  \kappa^\alpha\theta $ if $\be$ is a limit,

\item $ \fs{ \kappa^{\al+1}   }{\theta}_\kappa =  \kappa^\al \theta $,

\item $ \fs{ \kappa^\alpha(\be+1)   }{\theta}_\kappa =  \kappa^\al \be +  {\fs {\kappa^\al}\theta_\kappa } $ if $\be>0$, and

\item $ \fs{ \kappa^\alpha }{\theta}_\kappa =  \kappa^{\fs \al\theta_\kappa } $ if $\al$ is a limit.

\end{enumerate}
\end{definition}

We write simply $\fs\al\theta$ when $\kappa=\om$.
We define the set of coefficients of $\al$ in coefficient $\kappa$-normal form by $\coeffsk\kappa 0 = \{0\}$, $\coeffsk\kappa {\kappa^{\alpha}\theta+\beta} = \coeffsk\kappa {\al} \cup \coeffsk\kappa \be \cup \{\theta\}$.
The maximal coefficient of $\alpha$ is given by $\mck\kappa\al=\max\coeffsk\kappa\al$; we omit the subindex $\kappa$ when $\kappa=\om$.

\begin{definition}\label{defTau}
Let $\xi<\ve(\kappa)$ be in $\kappa$-normal form.
The {\em terminal part} of $\xi$, denoted $\tau_\kappa(\xi)$ or $\tau(\xi)$ when clear from context, is given recursively by
\begin{enumerate}

\item $\tau(0) = 0$ and $\tau(\zeta+1) = 1$,

\item $\tau (\kappa^\al \be+\ga) = \tau  (\ga)$ if $\ga>0$,

\item $\tau (\kappa^\al \be ) = \be$ if $\be$ is a limit,

\item $\tau (\kappa^\al (\be+1) ) = \tau(\al)$ if $\al$ is a limit, and

\item $\tau(\kappa^ {\al+1}(\be+1)) = \kappa$.

\end{enumerate}
\end{definition}

It is readily checked that if $\tau = \tau(\xi)$ is a limit and $\theta<\tau $ then $\fs\xi\theta_\kappa <\xi$, and moreover $\xi = \lim_{\theta\to \tau} \fs\xi\theta_\kappa$.
Note that we have defined $\fs\xi\theta_\kappa$ even when $\theta\geq \tau$, for notational convenience.

The following is checked by induction on $\xi$.

\begin{lemma}\label{lemmFundProp}
Let $\lambda < \ve(\kappa)$ be a limit ordinal and $\theta<\kappa$.
\begin{enumerate}

\item\label{itBoundMC} $\theta\leq \mc {\fs\lambda\theta} \leq \max \{\mc \lambda,\theta  \}$.

\item\label{itFundPropMajor} If $\xi < \lambda$ and $\mck\kappa \xi < \theta  $, then $\xi < \fs\la\theta_\kappa$.

\item Whenever $ \theta < \eta<\kappa$, it follows that $\fs\la\theta_\kappa < \fs\la\eta$.

\end{enumerate}
\end{lemma}

The fundamental sequences admit a sort of left inverse, given by the following operation.

\begin{definition}\label{defCeil}
Define $\lceil \xi\rceil_\kappa$ where, $\xi $ is written in $\kappa$-normal form, by 
\begin{enumerate}

\item $\lceil 0 \rceil_\kappa = 0$,

\item $\lceil \kappa^\al \be+\ga\rceil_\kappa = \kappa^\al \be+\lceil \ga\rceil_\kappa$ if $\ga>0$,

\item $\lceil \kappa^\al \be \rceil_\kappa = \kappa^{\al+1}$ if $\be > 1$, and

\item $\lceil \kappa^\al  \rceil_\kappa = \om^{\lceil \al  \rceil_\kappa}$.

\end{enumerate}
\end{definition}

As before, we omit the subindex when $\kappa=\om$.

\begin{lemma}\label{lemmCeilProp}
If $\theta \in(1,\kappa)$ and $\xi$ is a limit with $\tau(\xi) = \kappa$ then $\lceil \fs\xi \theta_\kappa \rceil_\kappa = \xi$.
\end{lemma}

\begin{proof}
We sketch the proof, which proceeds by induction on $\xi$.
The critical case is where $\xi = \kappa^{\al}  $.
If $\al$ is a limit, then $\fs{\kappa^\al }\theta_\kappa =  \kappa^{\fs\al\theta_\kappa} $, so that  $\lceil\fs{\kappa^\al }\theta_\kappa \rceil_\kappa =  \kappa^{\lceil \fs\al\theta_\kappa\rceil_\kappa} $.
By induction hypothesis, $\lceil \fs\al\theta_\kappa\rceil_\kappa = \al$, so that 
$\lceil\fs{\kappa^\al  }\theta_\kappa \rceil_\kappa =  \kappa^\al$.
Otherwise, $\al=\al'+1$ and $\fs{\kappa^{\al'+1}  }\theta_\kappa = \kappa ^{\al'}\theta $.
So, $\lceil \fs{\kappa^{\al'+1}  }\theta_\kappa \rceil_\kappa = \kappa^{\al'+1}$.
\end{proof}

As a corollary, we obtain that fundamental sequences are injective in the following sense.

\begin{corollary}\label{corFSInj}
If $\tau_\kappa(\al) = \tau_\kappa (\al') = \kappa$ and $\theta,\theta'\in (1,\kappa)$ are such that $ \fs\al\theta_\kappa = \fs{\al'}{\theta'}_\kappa$, then $\al=\al' $. 
\end{corollary}

\begin{proof}
By Lemma \ref{lemmCeilProp}, we have that
$\al = \lceil \fs{\al}{\theta}_\kappa\rceil_\kappa = \lceil \fs{\al'}{\theta'}_\kappa\rceil_\kappa  = \al'$.
\end{proof}

Next we turn our attention to $\kappa=\om$. Here, the {\em Bach\-mann property} holds for the system of fundamental sequences.
It is convenient to define this notion with some generality.

\begin{definition}
Let $\Lambda$ be a countable ordinal. A {\em system of fundamental sequences on $\Lambda$} is a function $\cdot [\cdot] \colon \Lambda \times \mathbb N \to \Lambda$ such that
\begin{enumerate}

\item  $  \alpha [n] \leq \alpha$ with equiaity holding if and only if $\alpha = 0$,

\item $\alpha[n] < \alpha[m]$ whenever $n< m$, and 

\item $\lambda  = \displaystyle \lim_{n\to \infty} \lambda[n]$ whenever $
\lambda $ is a limit.

\end{enumerate}
The system of fundamental sequences has the {\em Bach\-mann property} if whenever $\alpha[n] < \beta < \alpha$, it follows that $\alpha[n] \leq \beta[1]$.
\end{definition}

The fundamental sequences we have defined are known to enjoy the Bach\-mann property~\cite{Schmidt77}.

\begin{lemma}
If $\la<\ve_0$ is a limit and $\xi \in \big( \fs\la n ,\lambda\big )$, it follows that $\fs\la n \leq \fs\xi 1$.
\end{lemma}

Finally, we want to observe that if $\fs \be q\leq \al\leq \be$, then we can obtain information about the coefficients of $\be$ from those of $\al$.
To make this precise, we first define a {\em truncation} of $\al$ to be any $\be\leq \al$ such that $\coeffs\be\subseteq \{0,1\} \cup \coeffs \al$.

\begin{lemma}\label{lemTrunc}
Let $\al,\be<\ve_0$ with $\be$ a limit.
If $\fs \be q\leq \al\leq \be$, there is a truncation $\al'$ of $\al$ and some $q'\geq q$ such that $\al'=\fs\be{q'}$.
\end{lemma}

\begin{proof}
By induction on $\be$.
Write $\be = \om^{\be_1} b +\be_0 $ and $\al = \om^{\al_1} a +\al_0 $ in $\omega$-normal form and consider the following cases.
\begincases

\item ($\be_0 > 0$).
Then we must have $\al = \om^{\be_1} b + \al_0 $ with $\fs {\be_0} q\leq \al_0 < \be_0$, which by the induction hypothesis yields a truncation $\gamma$ of $\al_0$ and $q'>q$ such that $\gamma=\fs{\be_0}{q'}$.
The desired truncation of $\al$ is then $\om^{\be_1} b +\ga$.

\item ($\be_0 = 0$).
Consider the following sub-cases.
\beginsubcases

\item ($b>1 $).
Then, $\al_1=\be_1$ and $b=a+1$.
Moreover, $\fs {\om^{\be_1}} q\leq \al_0 < \om^{\be_1}$, yielding a truncation $\gamma$ of $\al_0$ and $q'>q$ such that $\gamma=\fs{\om^{\be_1}}{q'}$.
It follows that $\om^{\al_1} a +\ga$ is the desired truncation of $\al $.

\item ($b = 1 $).
We consider two further sub-cases.
\beginsubcases

\item ($\be_1 = \de+1$).
Then, $\fs{\be}q = \om^\de q$, and from $\fs\be q\leq \al <\be$ we see that $\al = \om^{\be_1}q' + \al_0$ for some $q'\geq q$.
In this case, $\om^{\be_1}q' $ is a truncation of $\al$ and $\om^{\be_1}q'  = \fs\be{q'} $.

\item ($\be_1$ is a limit).
We have that $\fs {\om^{\be_1}} q\leq \al < \om^{\be_1}$, so that $\fs {\om^{\be_1}} q\leq \al_1 < \om^{\be_1}$.
Let $\gamma$ and $q'$ be the truncation and number given by the induction hypothesis for $\al_1$.
Then, $\om^\ga$ is the corresponding truncation of $\al$.\qedhere
\end{enumerate}
\end{enumerate}

\end{enumerate}
\end{proof}

\section{Parametrized fast-grow\-ing hierarchies}\label{secAFun}

We may use fundamental sequences to define very large natural numbers in terms of ordinals below $\ve_0$.
We introduce a version with an extra parameter $k$, which will serve as the `base' in our fast Goodstein walks.

\begin{definition}\label{defA}
For $2\leq\la\leq\om$, we define $\A_\la\colon \ve(\la^+)\to \la^+$ as follows.
First, introduce the abuse of notation $\A_k(\al-1) = \mc\al$, when $\al $ is a limit or zero ($\al-1$ itself remains undefined when $\al$ is a limit).
Suppose inductively that $\A_\la(\xi)$ is defined for all $\xi<\al$ and define $\A_\la(\xi)$ according to the following cases.

\begincases
\item ($\la<\om$).
We divide in two sub-cases.
\beginsubcases
\item ($\xi<\om$). Set $\A(\xi)=\xi +1$.

\item ($\xi\geq \om$). Write $\xi=\al+b$ with $\al$ a limit.
Define $\A_\la^{(i)}(\xi)$ recursively by
\begin{enumerate}
\item $\A_\la^{(0)}(\xi) = \A_\la(\xi-1)$ and

\item $\A_\la^{(i+1)}(\xi) = \A_\la(\fs\al{\A_\la^{(i+1)}(\xi) } )$.
\end{enumerate}
Then, set $\A_k(\xi) =  \A_k^{(k)} (\xi)$.

\item ($\la=\om$). $\A_\om(\xi)$ is inductively defined to be the least $\theta>\mck\Om \xi$ such that if $\zeta<\xi$ and $\mck \Om\zeta<\theta$ then $\A_\om(\zeta)<\theta$.
We further define $\A_\om^{(i)}(\xi) = \A_\om(\xi)$ for all $i$.
\end{enumerate}
\end{enumerate}

\end{definition}
We remark that $\A_\om$ is a variant of the $\vartheta$ function~\cite{BuchholzOrd}, except that $\A_\om(\xi)$ may take values that are not  principal numbers; this version of the $\vartheta$ function is denoted $\sigma$ in \cite{FWTheta}, where it is studied in some detail.
We also remark that our somewhat trivial definition of $\A_\om^{(i)}(\xi) $ will be useful for treating the cases $\la<\om$ and $\la=\om$ uniformly, for example in Lemma \ref{lemmMajorize}.

The main characterization we use of the $\A_\om$ function is the following.

\begin{proposition}\label{propAOrder}
If $\zeta,\xi<\ve_{\Om+1}$, then $\A_\om(\zeta)<\A_\om(\xi)$ if and only if $\zeta<\xi$ and $\mck\Om\zeta<\A_\om(\xi)$.
\end{proposition}

Throughout the paper often write $\A (\xi) = \A_\lambda(\xi)$ whenever $\lambda$ is fixed and clear from context.

\begin{definition}\label{defPredek}
For $k\in[2,\om ) $ and $\xi = \alpha +b\in (0,\varepsilon_0)$ with $\al$ a limit, we define
$\predek \xi k =  \fs\al{\A^{(k-1)}(\xi)} $.
\end{definition}

It is readily checked that Definition \ref{defA} yields
\begin{equation}\label{eqRecurs}
\A_k(\xi) = \A_k(\predek \xi k)
\end{equation}
for all $\xi$.

The values $\A_k(\xi)$ grow rather quickly, but the following simple lower bound will be useful to us.
Recall that if $\al$ is zero or a limit, then $\A(\al-1) =\mc\al$ by definition.

\begin{lemma}\label{lemmBoundTwo}
Let $k\in[2,\om)$ and $\xi <\ve_0$ and write $\A$ for $\A_k$.
\begin{enumerate}

\item \label{itBoundTwoTwo} If $\xi>\om$ and $i<j<\om$, then $\A(\xi-1) \leq  \A ^{(i)}(\xi) < \A ^{(j)}(\xi)$.

\item \label{itBoundTwoOne} $  \mc \xi \leq \A(\xi-1) <\A (\xi)$.

\item \label{itBoundTwoThree} If $\xi >\om$ and $k \in [2,\om)$, then $\A_k(\xi) > 2\mc\xi +2$.

\end{enumerate}

\end{lemma}

\begin{proof}
Proceed to prove the three claims simultaneously by induction on $\xi = \al+b$, where $\al$ is zero or a limit.
\medskip

\noindent {\em Proof of Claim \ref{itBoundTwoTwo}.}  We have that $\A^{(0)} (\xi) = \A(\xi-1)$, and induction hypothesis applied to $\fs\al{\A^{(i)} (\xi) } < \xi$ yields
\[\A^{(i+1)} (\xi) = \A( \fs\al{\A^{(i)} (\xi) } ) >\mc{\fs\al{\A^{(i)} (\xi) }} \geq \A^{(i)} (\xi) .\]
Induction on $i,j$ then yields the claim.
\medskip

\noindent {\em Proof of Claim \ref{itBoundTwoOne}.} If $\xi<\om$, $\xi=\mc\xi =\A(\xi-1)$ and $\A(\xi) = \xi+1 > \xi$.
Otherwise, $\al$ is a limit.
That $\mc\xi\leq \A(\xi-1)$ follows by definition when $b=0$, otherwise we may apply the induction hypothesis to the third claim to see that $\mc\xi < 2\mc{\xi-1} +2 < \A(\xi-1)$.
That $\A(\xi-1)<\A(\xi)$ follows from direct computation when $\xi=\om$, as $\A(\xi-1) = 1 <k+1 =\A(\om)$.
If $\xi>\om$, the fist claim yields $\A(\xi-1) = \A ^{(0)}(\xi) \leq \A ^{( k )}(\xi) = \A(\xi)$.
\medskip

\noindent {\em Proof of Claim \ref{itBoundTwoThree}.}   Consider two cases.
\begincases
\item ($\al=\om$).
Proceed by induction on $b$. 
If $b=1$ then $\A^{(0)} (\xi) = \A(\om ) = k+1\geq 3  $, so inductively  $\A^{(i)} (\xi) \geq 3+i$, yielding $\A(\xi) = \A^{(k)}(\xi) = 3+k\geq 5 = 2 \mc \xi + 3$.
Otherwise, the induction hypothesis yields $\A^{(0)} (\xi) \geq 2(b-1) + 3 = 2b + 1$, so $\A^{(k)}(\xi) \geq 2 b + k + 1 \geq 2b+3 = 2\mc\xi+3 $.

\item ($\al>\om$).
We have by the second claim that $\A^{(1)}(\xi)  > \mc{\xi} \geq 1 =\mc\om$.
Since $\al>\om$ and $\A^{(1)}(\xi) > \mc\om$, we have that $\fs{\al}{\A^{{(1)} }(\xi) } >\om$, so we may apply the induction hypothesis to conclude that
\[\A(\xi) \geq \A^{(2)} (\xi) =\A(\fs{\al}{\A^{{(1)} }(\xi) }) > 2\mc { \fs{\al}{ \A^{{(1)} }(\xi) } } +2 > 2\mc\xi+2.  \qedhere\]
\end{enumerate}
\end{proof}

Recall below that we have defined $\A^{(i)}_\om(\xi) = \A _\om(\xi)$ for all $i$.
This will allow us to write the following uniformly for all $\lambda\leq \om$.

\begin{lemma}\label{lemmMajorize}
Let $2\leq k< \lambda \leq \om$ and $\al< \be<\ve(\la^+)$.
If $\mck{\la^+} \al < \A^{(k-1)}_\la(\be)$, then $\A_\la(\al)<\A_\la(\be)$.
\end{lemma}

\begin{proof}
By induction on $\be$.
First assume $\la<\om$.
Write $\be=\la+t$ for $\la$ a limit.
If $\al=\la+t'$ for some natural number $t'$, then $t'<t$ and the claim is established by an easy induction on $t$.
Otherwise, using Lemma \ref{lemmFundProp}.\ref{itFundPropMajor}, we see that $\al\leq \lfloor \be\rfloor_k$, and moreover Lemma \ref{lemmBoundTwo}.\ref{itBoundTwoTwo} yields
\[\A^{(k-1)}(\predek \be k) > \mc {\lfloor \be\rfloor_k} \geq \A^{(k-1)}(\be) >\mc\al.\]
Thus we may apply the induction hypothesis to see that $\A(\al) <\A(\lfloor \be\rfloor_k ) = \A(\be)$.

If $\la=\om$, this simply repeats the characterization we have given for establishing  $\A_\om(\al)<\A_\om(\be)$, using the equality $\A^{(k-1)}(\be) = \A(\be)$.
\end{proof}

The following is an essential application of the Bachmann property to the study of the $\A$ functions.

\begin{lemma}\label{lemmBetweenSlow}
Fix $k\geq 0$ and let $\A = \A_k$.
Let $\lambda<\ve_0$ be a limit.
Then, if $ \fs\lambda q   < \zeta < \lambda$, it follows that $\A(\zeta) > \A( \fs\lambda q)$.
\end{lemma}

\proof
Write $\zeta = \alpha +d$ and $\fs\lambda q =\be+c $ with $\al,\be$ limits or zero and proceed by induction on $\zeta$.
Note that either $\al=\be $ or else $\fs \lambda q < \alpha <\lambda$, from which it follows using the Bach\-mann property that $\fs\lambda q \leq \fs \alpha 1 < \predek\zeta k$.
With this in mind, consider two cases.

\begincases

\item ($\al=\be $).
Then,
\[\al+c = \fs\lambda q    < \zeta = \al +d, \]
so $ c < d $ and $\A(\al + c)  < \A(\al +d )  $ by Lemma \ref{lemmMajorize}.

\item ($ \predek \zeta k > \fs \lambda q  $).
In this case, $  \fs \lambda q < \predek \zeta k <\zeta < \lambda$, so we see by the induction hypothesis that
$\A( \fs\lambda q  )  < \A(\predek \zeta k ) = \A(\zeta)$.\qedhere
\end{enumerate}
\endproof

\begin{corollary}\label{corBetweenSlow}
Fix $k\geq 0$ and let $\A = \A_k$.
Let $\xi =  \lambda +b<\varepsilon_0$ with $\lambda$ a limit.
Then, if $\lfloor\xi\rfloor_k  < \zeta <  \lambda$, it follows that $\A(\zeta) > \A(\xi)$.
\end{corollary}

\begin{proof}
Note that $\lfloor\xi\rfloor_k = \fs\lambda q $ with $q = \A^{(k-1)} (\xi)  $.
Thus Lemma \ref{lemmBetweenSlow} yields $\A(\xi) = \A(\lfloor\xi\rfloor_k) < \A(\zeta)$.
\end{proof}
 
Thus Corollary \ref{corBetweenSlow} tells us that there may be cases where $\xi<\zeta$ yet $\A(\xi)>\A(\zeta)$.
However, what we {\em can} guarantee when $\xi<\zeta$ is that $\A(\zeta)$ will never lie between $\A(\xi)$ and $\A(\xi+1)$: 

\begin{lemma}\label{lemmOneMore}
Given $\xi<\zeta<\varepsilon_0$, $ \A(\zeta) \not\in \big  ( \A(\xi ), \A(\xi+1) \big )$.
\end{lemma}

\begin{proof}
Write $\xi = \alpha +a$ and $\zeta= \beta+b$ with $\al,\be$ limits.
Proceed by induction on $ \zeta $ and consider three cases.
\begincases

\item ($ \xi  \geq \beta$).
Then, since $\xi<\zeta$ we have that $\alpha=\beta$ and $a<b$, so that $\A(\xi+1) = \A(\al+a+1)   \leq \A(\al+b) = \A(\zeta)$.

\item ($\lfloor\zeta\rfloor_k \leq \xi $).
By Corollary \ref{corBetweenSlow}, $\A(\zeta) \leq \A(\xi ) $.

\item ($\lfloor\zeta\rfloor_k > \xi$).
In this case the claim is immediate from the induction hypothesis.\qedhere
\end{enumerate}
\end{proof}

Now that we have established the basic properties of the $\A $ functions, it is time to define normal forms for natural numbers based on them.

\section{Normal forms based on the fast-grow\-ing hierarchy}\label{secNF}

Maximal Goodstein processes for the $\A$ functions are obtained by first identifying base-change maximal normal forms.
We define them below, with maximality proven in subsequent sections.
Below, for an expression $\varphi( \xi ) $ we set $\xi_\ast = \nu \xi.\varphi$ if either $\xi_\ast$ is the maximum element of $\varepsilon_0$ satisfying $\varphi(\xi_\ast)$, or $\xi_\ast = 0$ and no such maximum exists.

\begin{definition}\label{defNF}
Fix $k\in [2,\om)$ and let $\A ( \xi ) = \A_k ( \xi)$.
Given $m \in \N$, we define $\A  ( \xi ) $ to be the {\em $k$-normal form} of $m$, in symbols $m \gknf \A ( \xi ) $, if $m = \A ( \xi ) $ and there exist sequences $\xi_0 , \xi_1, \ldots, \xi_n  $ such that:
\begin{enumerate}

\item $\xi_0 = 0$.

\item Given $i<n$, $\xi_{i+1} = \nu \zeta. \big ( \A ( \zeta ) \leq m \text{ and } \mc{\zeta} \geq \A ( \xi_i ) \big ).$

\item $\A(\xi_n) = m$.

\end{enumerate}
We call $(\xi_i)_{i\leq n}$ the {\em normal form sequence} of $m$.
\end{definition}

Let us introduce some notation for normal form sequences.
Suppose that $m$ has normal form sequence $ \big ( \A (\xi_i) \big )_{i\leq n}$.
Then, we define $\penum m = m_{n-1}$, and $\penum \xi = \xi_{n-1}$.
It is easy to see that $\penum m\gknf \A(\penum \xi)$, with the same normal form sequence as $m$ but truncated at the second to last element.

We often need results of the form {\em If $\A(\xi)$ is in normal form, then $\A(\zeta)$ is also in normal form,} where $\zeta$ is `similar' to $\xi$ in some way.
The following general principle will be useful for establishing this.

\begin{lemma}\label{lemmGeneralNormalForm}
If $\A(\xi)$ is in normal form and $\zeta $ is such that
\begin{enumerate}[label=(\alph*)]

\item $ \A(\zeta) <\A(\xi+1)$,

\item\label{condGenTwo} $\A(\xi)\leq \mc {\zeta} $, and

\item\label{condGenThre} for all $\theta$, if $\mc {\theta}\geq\A(\xi)$ and $\A(\theta) \leq \A(\zeta)$, then $\theta \leq \zeta$,

\end{enumerate}
then $\A(\zeta)$ is in normal form.
\end{lemma}

\begin{proof}
Let $(\xi_n)_{i\leq n}$ be the normal form sequence for $\A(\xi)$.
Note that $\A(\xi) \leq \mc {\zeta} < \A(\zeta) <\A(\xi+1)$ implies that $\xi \geq \om$ (otherwise $\A(\xi+1)= \A(\xi) + 1$), so $n>0$.
Note also that the maximality of $\xi_i$ yields $\xi_{i+1} \leq \xi_i$ for $0<i < n$, and $\xi_{i+1} = \xi_i$ is impossible since $\A(\xi_{i+1}) > \mc{\xi_{i+1}} \geq \A(\xi_i)$.
It follows that $\xi_{i+1}>\xi_n$ for all $i<n -1$.

Let $(\zeta_r)_{i\leq r}$ be the normal form sequence for $\A(\zeta)$.
We claim that $r=n+1$, $\zeta_i = \xi_i$ for $i\leq n$, and $\zeta_r = \zeta$, witnessing that $\A(\zeta)$ is in normal form.
If $i<n$, $\A(\xi_{i+1}) \leq \A(\xi) < \A(\zeta) $.
If $\delta>\xi_{i+1}$ and $\mc\delta\geq \A(\xi_i)$, the maximality of $\xi_{i+1}$ yields $\A(\delta) >\A(\xi )$.
Moreover, $\delta>\xi_{i+1}>\xi$, so Lemma \ref{lemmOneMore} yields $\A(\delta) \geq \A(\xi+1) > \A(\zeta)$.
Thus by induction on $i$ we see that $\zeta_i = \xi_i$, since $\xi_{i+1}$ is maximal with $\A(\xi_{i+1}) \leq \A(\zeta)$ and $\mc{\xi_{i+1}} \geq \A(\xi_i) \stackrel{\text{\sc ih}} = \A(\zeta_i)$.
Conditions \ref{condGenTwo} and \ref{condGenThre} then guarantee that $\zeta_{n+1} = \zeta$, as needed.
\end{proof}

Below we provide some applications of this general lemma.

\begin{lemma}\label{lemmNFBminus}
If $\xi$ is a successor and $m = \A(\xi)$ is in normal form, then either
\begin{enumerate}[label=(\alph*)]

\item $n = \A(\xi-1)$ is in normal form, or

\item $\xi = \al + \A(\penum \xi)$, where $\al$ is a limit or zero and $\mc\al <  \A(\penum \xi)$.
 
\end{enumerate}
\end{lemma}

\begin{proof}
First, note that $\A(\xi-1) < \A(\xi) <\A(\penum \xi+1 )$.
We check that $\xi-1$ satisfies the maximality condition for normal form sequences.
If $\A(\theta) \leq \A(\xi-1)$ and $\mc\theta\geq \A(\penum \xi)$, then we also have that $ \A(\theta) \leq \A(\xi)$, which by maximality of $\xi$ implies that $\theta\leq \xi$. But $\A(\xi)>\A(\xi-1)$, so $\theta\neq\xi$, hence $\theta\leq\xi-1$.
Thus $\A(\xi-1)$ is in normal form unless $\mc{\xi-1} <\A(\penum \xi)$.
The only way to have that $\mc\xi\geq \A(\penum \xi)$ but $\mc{\xi-1}< \A(\penum \xi)$ is that $\xi=\al+\A(\penum \xi)$ with $ \mc{\al} < \A(\penum \xi)$, as needed.
\end{proof}

\begin{lemma}\label{lemmAiNorm}
Let $k\in [2,\om)$ and write $\A$ for $\A_k$.
Suppose that $\A(\xi)$ is in normal form and let $i < k-1$.
Then, $ \A( \fs\xi{\A^{(i)}(\xi)} ) $ is in normal form.
\end{lemma}

\begin{proof}
Write $\xi=\al+b$ with $\al$ a limit and let $m=\A(\xi)$, $q=\A^{(i)}(\xi)$, and $m' = \A( \fs\xi{\A^{(i)}(\xi)} ) $.
Consider two cases.

\begincases
\item ($b=0$ or $\mc{\xi-1} < \penum m$).
In order to apply Lemma \ref{lemmGeneralNormalForm}, we note that $  \A( \fs\alpha {q }  ) < \A(\xi ) < \A(\penum \xi + 1) $, where the first inequality uses Lemma \ref{lemmBoundTwo}.\ref{itBoundTwoTwo}.
Lemma \ref{lemmBoundTwo}.\ref{itBoundTwoTwo} also yields $q \geq \mc \xi \geq \penum m$.
Thus it suffices to show that if $\theta$ is such that $\mc\theta\geq \penum m$ and $\A(\theta)\leq m'$, then $\theta\leq \fs\al q $.
If $\theta \geq  \xi$ then $\A(\theta) > m >m'$ by maximality of $\xi$.
If $\theta = \alpha + p$ then $p < b$.
Consider two sub-cases. If $b=0$, no such $p$ exists.
If $\mc{\xi-1} < \penum m$ holds, this implies that $\mc{\theta} < \penum m$, contrary to our choice of $\theta$.
Finally, we note that if $\theta \in ( \fs\al q  , \al)$, Lemma \ref{lemmBetweenSlow} yields $\A(\theta)>m'$.
We thus conclude that $\A(\fs\alpha {q })  $ is in normal form.

\item ($b>0$ and $\mc{\xi-1} \geq \penum m$).
Lemma \ref{lemmNFBminus} shows that $\A(\xi -1)$ is in normal form.
We use this to check that $ \fs\alpha {q } $ is in normal form, once again using Lemma \ref{lemmGeneralNormalForm}. Using Lemma \ref{lemmBoundTwo}, we can see that $\A(\xi-1) \leq q \leq \mc{\fs\al q}$ and $\A(\fs\al q)<\A(\xi)$.
Thus it remains to check that if $\chi$ is such that $\mc\chi \geq \A(\xi-1)$ and $\A(\chi)\leq m'$, then $\chi\leq \fs\al q$.
Since $\A(\xi-1)$ is in normal form, we must have that $\chi \leq \xi-1$.
If $\chi =\al+p$ then $p < b $, so $\mc {\al+ p } < \A(  \al+ p )\leq \A(  \xi-1)$.
If $\chi\in (\fs\al q ,  \al)$, Lemma \ref{lemmBetweenSlow} yields $\A(\chi) > m'$.
Hence indeed $\A(\fs\al q )$ is in normal form.\qedhere
\end{enumerate}
\end{proof}

\section{Base change}\label{secBC}

The last ingredient we need in order to define our maximal Goodstein process is the base change operator.
As in the classical Goodstein process, $\bch m{\bases k\la}$ replaces every instance of $k$ by $\la$ in the normal form of $m$.

\begin{definition}
Given $2\leq k<\la\leq\om$ and $m\in \N$, we define the base change operation $\bch m {\bases k \la} = \chg m$ inductively as follows.
Let $\kappa=\la^+$, and set:
\begin{enumerate}

\item $\bch 0 {\bases k \la} = 0$.

\item For $m \gknf \A_k(\xi)  $ we set
$\chg m   = \A_\la (\chg\xi)$.

\item For $\xi = \omega^{ \al}b+\ga  $ in $\omega$-normal form, we set
$\chg\xi = \kappa^{ \chg \al}\chg b + \chg\ga$.

\end{enumerate}
\end{definition}

Base-change maximal normal forms lead to monotone base-change operators, in the following sense.

\begin{proposition}\label{propMaxToMon}
Let $2\leq k<\la\leq \om$ and write $\uparrow$ for $\bch{}{\bases k\la}$.
Suppose that $m$ has the property that for all $n<m$, if $n=\A_k(\zeta)$, then $\chg n \geq \A_\la(\chg \zeta) $.
Then, whenever $2\leq k<\la$ and $i<j<m $, it follows that $\bch i{\bases k\la} < \bch j{\bases k\la}$.
\end{proposition}

\proof
Working inductively, we may assume that $j=i+1$.
Then, we have that $i= i + 1 =   {\A_k (i)}$, and by base-change maximality,
\[ \chg j  \geq \chg{\A_k(i) } =  \A_\ell (\chg i  ) = \chg{ i  } +1>\chg i. \qedhere \]
\endproof

In fact, this monotonicity property is crucial for proving that Goodstein processes terminate; Proposition \ref{propMaxToMon} tells us that we have monotonicity for free, {\em if} we prove base-change maximality.

\begin{remark}\label{remMaxToMon}
Our goal is to prove that base-change maximality indeed holds.
If we are to prove this by induction on $m$, then Proposition \ref{propMaxToMon} tells us that we may assume that the base-change operator is monotone below $m$.
Thus this will be used as a hypothesis in many of the following results.
Once Theorem \ref{theoMax} and the subsequent Corollary \ref{corMon} (which states that the base change {\em is} monotone) have been established, we may drop this assumption wherever it was used.
\end{remark}

The next lemma is our first application of the proof strategy of Remark \ref{remMaxToMon}.

\begin{lemma}\label{lemmBchMonOrd}
Assume monotonicity of base change below $m$ and let $2\leq k<\la\leq \om$.
Suppose that $\xi,\zeta<\ve_0$ are such that $\mc\xi,\mc\zeta <m$.
Then,
\begin{enumerate}

\item If $\xi<\zeta$, then $\bch \xi{\bases k\la} < \bch \zeta{\bases k\la}  $.

\item If $\mc\xi<\mc\zeta < m$, then $\chg{\mc\xi} <\chg{\mc\zeta}$.

\end{enumerate}
\end{lemma}

\begin{proof}[Proof sketch.]
Both claims are easily verified using the fact that the base change is applied to the coefficients of $\xi$ and $\zeta$, and we may apply monotonicity to them given that they are bounded by $m$.
\end{proof}

\begin{lemma}\label{lemmOrdBch}
Assume monotonicity of base change up to $ \mc{\al}$, where $\al<\ve_0$ is a limit.
Let $n \in\N$, $2\leq k<\la\leq \om$, and write $\uparrow$ for $\uparrow^\la_k$.
Then,
\begin{enumerate}

\item $\mc{\chg{ \fs \al n } } \leq \mc{\fs{\chg\al  }{\chg n }_{\la^+}}$, and

\item $\chg{ \fs\al n }  \leq \fs{\chg\al }{\chg n}_{\la^+}.$

\end{enumerate}
\end{lemma}

\begin{proof}
By induction on $\al$, where the critical cases are $\al=\om^\be c $ with $c>1$ and $\al =\om^{\be+c}$ with $\be$ a limit and $c>0$.
In the first case, we have that
\[\chg{ \fs\al n } = \chg{\big (\om^\be (c-1) + \fs{\om^\be}n \big ) } = \om^{\chg \be} \chg {(c-1)} + \chg{\fs{\om^\be}n},\]
while
\[\fs{\chg\al }{\chg n}_{\la^+} = \fs{\om^{\chg \be}\chg{ c } }{\chg n}_{\la^+}  = \om^{\chg \be}   ( \chg{ c } - 1 ) + \fs{  \om^{\chg \be} }{\chg n}_{\la^+} .\]
Monotonicity below $\mc\al$ yields $\chg {(c-1)} \leq \chg c-1$, while the induction hypothesis yields $\chg{\fs{\om^\be}n} \leq \fs{  \om^{\chg \be} }{\chg n}_{\la^+}$.
From this it readily follows that $\chg{ \fs\al n } \leq \fs{\chg\al }{\chg n}_{\la^+}$, establishing the first item, and inspection on the coefficients involved establishes the second.
The case for $\al =\om^{\be+c}$ is similar.
\end{proof}

Note that as an immediate corollary we obtain that $\A_\la(\chg{\fs\al n}) \leq \A_\la(\fs{\chg \al}{\chg n})$.
The following claims are easily verified simultaneously by induction on $n$ and $\al$.

\begin{lemma}\label{lemmBchBigger}
Let $2\leq k<\la\leq \om$ and assume monotonicity of base change below $n$.

\begin{enumerate}

\item If $n \in \N$ then $n \leq \bch n {\bases k \la}$, and if $n =\A(\xi)$ with $\xi\geq \om$, then $2n +2 \leq \bch n {\bases k \la}$.

\item If $\alpha <\varepsilon_0$ and $\mc\al < n$ then $\alpha \leq \bch \alpha {\bases k \la} $, and if $\la<\om$ then $\bch \alpha {\bases k\la} \in {\rm Lim}$ if and only if $\alpha \in {\rm Lim}$.

\end{enumerate}
\end{lemma}

\begin{proof}
Write $\chg{}$ instead of $\bch{}{\bases k\ell}$, $\A$ for $\A_k$ and $\B$ for $\A_\la$.
For the first claim, proceed by induction on $n$.
Clearly, $0 \leq \chg 0$.
Otherwise, write $n \gknf \A(\xi) $, so that $\chg n = \B (\chg \xi)$. 
If $\xi$ is finite, then $\xi=n-1$ and the induction hypothesis yields $\chg \xi\geq \xi$, so that $\chg n = \chg \xi+1 \geq n$.
Otherwise, proceed by a secondary induction on $\xi$ to show that if $\xi\geq \om$ then $2\A(\xi) + 2 < \B(\chg\xi)$.
By induction on $i$, we show that $ \A^{(i)}(\xi) \leq \B^{(i)}(\chg \xi) $.
For $i=0$, we have if $\xi$ is a limit that $\A^{(0)}(\xi) = \mc\xi \leq \chg {\mc\xi} = \B^{(0)}(\chg\xi)$, where we have used Lemma \ref{lemmBoundTwo} to see that $\mc\xi<n$, so that we may apply the induction hypothesis.
Otherwise, write $\xi=\al+b$, so that
\[\A^{(0)}(\xi) = \A(\al+b-1) \leq \B(\chg\al +\chg(b-1)) \leq \B(\chg\al +\chg b-1 ) = \B^{(0)} (\chg\xi ), \]
where we have used monotonicity of base change below $n$ to conclude that $\chg(b-1) \leq \chg b-1$.

Now assume that $ \A^{(i)}(\xi) \leq \B^{(i)}(\chg \xi) $.
We see that
\begin{align*}
\A^{(i+1)}(\xi) & = \A(\fs \al {\A^{(i)}(\xi)}) \stackrel{\text{\sc ih}(\xi)}\leq \B( \chg{\fs \al {\A^{(i)}(\xi)} })\\
&\leq  \B( {\fs {\chg \al} {\chg{\A^{(i)}(\xi)}} }) \stackrel{\text{\sc ih}(i)}\leq \B( {\fs {\chg \al} {\chg{\B^{(i)}(\xi)}} }) = \B^{(i+1)}(\chg\xi),
\end{align*}
where first we use the induction hypothesis on $\fs \al {\A^{(i)}(\xi)} < \xi$, then the secondary induction hypothesis on $i$, and we have used Lemma \ref{lemmOrdBch} to see that $\B( \chg{\fs \al {\A^{(i)}(\xi)} }) \leq  \B( {\fs {\chg \al} {\chg{\A^{(i)}(\xi)}} })$.
Lemma \ref{lemmBoundTwo} then yields
\[2\A(\xi)+2 = 2\A^{(k)}(\xi) +2 \leq 2\B^{(k)}(\chg\xi) +2 < \B^{(k+1)} (\chg \xi) = \B(\chg \xi).  \]
The second claim follows by an easy induction on $\al$ using the first claim.
\end{proof}

In general we have that $\chg{\fs\al b} \leq \fs{\chg\al}{\chg b}$, but on occasion it would be useful for this to be an equality.
In such cases, we may instead use a variant of base change such that $\chg {\fs\al b} = \fs{\upsquiggly \al}{\chg  b}$, defined next.

\begin{definition}
Let $2\leq k<\la\leq \om$, and write $\uparrow $ for $\uparrow^\la_k$ and set $\kappa=\la^+$.
For $\xi<\ve_0$ in $\om$-normal form, define $\upsquiggly\xi = \upsquiggly ^\la_k\xi$ recursively by
\begin{enumerate}

\item $\upsquiggly 0 = 0$,

\item $\upsquiggly (\om^\al b+\ga) =\kappa^{\chg \al}\chg b +\upsquiggly \ga$ if $ \ga>0$,

\item $\upsquiggly \om^\al (b+1) = \kappa^{\chg \al}\chg b + \upsquiggly \om^\al$ if $b>0$,

\item $\upsquiggly  \om^{\al+1} = \kappa^{\chg \al+1} $,

\item $\upsquiggly \om^\al = \kappa^{\upsquiggly \al}$ if $\al$ is not a successor.

\end{enumerate}
\end{definition}

Note the `critical clause' $\upsquiggly  \om^{\al+1} = \kappa^{\chg \al+1} $, which is what differentiates $\upsquiggly$ from $\chg{}$.
This is what will make $\upsquiggly$ commute with fundamental sequences.
The operator $\upsquiggly$ is closely related to the operation $\lceil\cdot\rceil$ of Definition \ref{defCeil}, which we recall provides a left inverse to the fundamental sequences; in fact, as may be seen from the first item of the following lemma, we may have equivalently defined $\upsquiggly \xi = \lceil \chg{ \fs \xi 2}\rceil$.

Below, recall that $\tau_\kappa(\xi)$ is the terminal part of $\xi$, as given by Definition \ref{defTau}.

\begin{lemma}\label{lemmSquigProp}
Let $\xi<\ve_0$ be a limit ordinal, $ t \in (1,\om) $ and $2\leq k<\la\leq \om$.
Write $\upsquiggly$ for $\upsquiggly^\la _k$ and $\chg{}$ for $\chg{}{\bases k\la }$.
Then,
\begin{enumerate}

\item $\upsquiggly \xi = \lceil \chg{ \fs \xi t}\rceil$.

\item \label{itSquigComm} $\chg{\fs\xi t} = \fs{\upsquiggly\xi}{\chg t} $.

\item If $t>1$ and $\zeta $ is a limit with $\chg{\fs\xi t} = \fs{\zeta}{\chg t} $, then $\zeta=\upsquiggly \xi$.

\item\label{itSquigVs} $\A_\la (\upsquiggly\xi+t ) \leq \A_\la (\chg\xi+t ) $.

\item $\xi$ is a limit iff $\upsquiggly\xi$ is a limit.

\item\label{itSquigFive} $\tau_{\la^+}(\upsquiggly \xi) ={\la^+}$.

\end{enumerate}
\end{lemma}

\begin{proof}[Proof sketch.]
Let $\kappa = \la^+$.
The first two items follow by induction on $\xi$.
The critical case is where $\xi=\om^{\al+1}$.
For the first claim, $\fs\xi t = \om^\al t$, so that $\chg{\fs\xi t} =  \kappa^{\chg\al}\chg t$, and
\[ \lceil \chg{\fs \xi t}\rceil_\kappa  = \lceil \kappa^{\chg\al}\chg t\rceil_\kappa = \kappa^{\chg\al+1} = \upsquiggly \xi .\]
For the second, we have that
\begin{align*}
\chg{\fs \xi t} & = \chg{\om^\al t} = \kappa^{\chg \al}\chg t = \fs{\kappa^{\chg \al + 1}}{\chg t}_\kappa =
\fs{\upsquiggly \om^{\al+1}}{\chg t}.
\end{align*}
Other cases follow by applying the induction hypothesis to the relevant sub-terms.

The third claim follows from Corollary \ref{corFSInj}, since from  $\fs{\zeta}{\chg t}  = \chg{\fs\xi t}  = \fs{\upsquiggly \xi}{\chg t}$ and injectivity of the fundamental sequences we obtain $\zeta= \upsquiggly \xi$.
The fourth follows from the fact that $\upsquiggly\xi+t  \leq  \chg\xi+t $ and $\mc{\upsquiggly\xi+t}  \leq  \mc{\chg\xi+t}$, which are verified by a routine induction, and the fifth by induction and case-by-case inspection.

The sixth item also proceeds by induction.
We treat the cases where $\xi = \om ^ {\al+1}$ and $\xi = \om ^ {\al} (b+1) $ with $b>0$.
We have that $\upsquiggly \om ^ {\al+1} = \kappa^{{\chg\al}+1}$,
and $\tau_\kappa(\kappa^{\chg\al+1}) = \kappa$ by definition.
Similarly, $\upsquiggly \om ^ {\al} (b+1) = \kappa^{\chg\al}\chg b + \upsquiggly \om^\al $, and by the induction hypothesis, $\tau_\kappa(\upsquiggly \om^\al) = \kappa$, so that $\tau_\kappa(\xi)=\kappa$ as well.
\end{proof}

The operation $\upsquiggly$ allows us to solve the equation $\chg {\fs\alpha c} = \fs{\theta}{\chg c}_{\la^+}$ when $\alpha$ and $c$ are given, but $\theta$ is unknown.
Now, suppose that we are given $\theta$ as well as the value of $\gamma:= {\fs\alpha c}$, but not $\alpha$ itself.
The following lemma provides conditions under which such and $\alpha$ can be found; in these cases, we have $\alpha = \lceil \gamma\rceil$.
Moreover, $\alpha$ additionally satisfies $\upsquiggly \alpha  = \theta$.
Let us make this precise.

\begin{lemma}\label{lemmCeil}
Let $2\leq k<\la\leq \om$ and write $\uparrow$ for $\uparrow^\la_k$ and $\upsquiggly$ for $\upsquiggly^\la_k$.
If $c>1$, $\ga<\ve_0$, and $\theta < \ve(\lambda^+)$ are such that $\tau_\kappa(\theta) =\kappa$ and $\fs \theta {\chg c}_{\la^+} = \chg \gamma$, then $\fs{\lceil \gamma\rceil} c=\gamma$ and $\upsquiggly \lceil \gamma\rceil = \theta$.
\end{lemma}

\begin{proof}
Let $\kappa=\la^+$.
We have that $\fs{\lceil \gamma\rceil} c=\gamma$ if and only if $\chg{\fs{\lceil \gamma\rceil} c}=\chg\gamma$ (because $\chg{}$ is strictly monotone by Corollary \ref{corMon}), if and only if $\fs{\upsquiggly \lceil \gamma\rceil}{\chg c}_\kappa = \chg\gamma$ (because $\chg{\fs{\lceil \gamma\rceil} c }= \fs{\upsquiggly \lceil \gamma\rceil}{\chg c}_\kappa$ by Lemma \ref{lemmSquigProp}.\ref{itSquigComm}), if and only if $\fs{\upsquiggly \lceil \gamma\rceil}{\chg c}_\kappa = \fs{\theta}{\chg c}_\kappa $ (by assumption on $\gamma$), if and only if $\upsquiggly \lceil \gamma\rceil = \theta$ (by Corollary \ref{corFSInj} and the fact that $\tau_\kappa(\upsquiggly \lceil \gamma\rceil  ) =\kappa$ by Lemma \ref{lemmSquigProp}.\ref{itSquigFive}).
Thus we prove that $\upsquiggly \lceil \gamma\rceil = \theta$ by induction on $\theta$.
Write $\theta=\kappa^\al \delta+\be$ in $\kappa$-normal form and $\gamma= \om^{\check\al} d+ \check\be$ in $\omega$-normal form, so that $\chg\gamma = \kappa^{\chg{\check \al}} \chg d +\chg{\check \be}$.
\begincases

\item ($\be>0$). Then $\fs\theta{\chg c}_\kappa = \kappa^\al d+ \fs\be{\chg c}_\kappa $, so that the assumption that $\fs \theta{\chg c} _\kappa = \chg \gamma $ yields $\chg {\check \al}=\al$, $\chg {d}=\de$, and $\chg{\check \be}= \fs\be{\chg c} $.
Since $\chg{\check \be}= \fs\be{\chg c} $, we may apply the induction hypothesis to obtain $\be=\upsquiggly \lceil\check\be\rceil$, and hence
\[\upsquiggly \lceil \ga \rceil = \chg{\om^{\check \al} d}+\upsquiggly \lceil \check \be\rceil = \kappa^\al \de + \be= \theta. \]

\item ($\be=0$ and $\de =\eta+1>1$).
Then $\fs\theta{\chg c}_\kappa = \kappa^\al \eta + \fs{\kappa^\al } {\chg c}_\kappa $, and $\chg{\check \al}=\al $, $\chg{d} =\eta$, and $\chg {\check\be} = \fs{\om^\al } {\chg c} $.
The induction hypothesis yields $\om^\al  = \upsquiggly \lceil \check\be\rceil $, hence
\[\upsquiggly \lceil \ga \rceil = \chg{\om^{\check \al} d}+\upsquiggly \lceil \check \be\rceil = \om^\al \eta+ \om^\al= \theta. \] 

\item ($\be=0$, $\delta =1$, and $\al = \chi+1$).
Then $\fs\theta{\chg c} = \om^\de  {\chg c} $, and $\chg{\check \al}=\chi $, $\chg{d} =\chg c$ so that $d=c>1$, and $\chg {\check\be} = 0 $, so $\check \be=0$.
Then, $\lceil\gamma\rceil = \om^{\check \al+1}$, so that
\[
\upsquiggly \lceil \ga\rceil  = \kappa^{\chg {\check \al} + 1} = \kappa^{\chi+1} =  \theta.
\]

\item ($\be=0$, $\delta =1$, and $\al$ is a limit).
Then $\fs\theta{\chg c} = \om^  {\fs{ \al } {\chg c} }$, and $\chg{\check \al}=\fs{ \al } {\chg c} $, $ d = 1$, and $ {\check\be} = 0 $.
The induction hypohtesis yields $\al  = \upsquiggly \lceil \check\al\rceil  $, hence since $\al $ is a limit, so is $ \lceil \check\al\rceil$, and
\[
\upsquiggly \lceil \ga\rceil  = \upsquiggly \om^{\lceil \check \al\rceil } = \kappa^{\upsquiggly \lceil \check \al\rceil} = \kappa^{\al }=\theta.\qedhere
\] 
\end{enumerate}
\end{proof}

We have noted that in general $\chg{\fs\xi n} \neq \fs{\chg\xi}{\chg n}_{\la^+}$, but there is an important case where this equality does hold.
Below, $\tau=\tau_\Om$.

\begin{lemma}\label{lemmUpSame}
Write $\chg{}$ for $\bch{}{\bases k\om}$ and let $\xi<\ve_{\Om+1}$.
If $\tau(\chg\xi) =\Om$, then $\chg{\fs\xi n} = \fs{\chg\xi}{\chg n}_{\la^+}$.
\end{lemma}

\begin{proof}[Proof sketch.]
This follows by induction, where the critical case is when $\xi = \om^\zeta b $.
We cannot have that $\chg b \in {\rm Lim}$, since this would imply that $\tau(\chg\xi) =\chg b <\Om $.
Hence $\chg b= \be+1$ for some $\be$, which means that $b\gknf \A(b-1)$ (as $\A(\zeta)$ is a limit for any $\zeta\geq \om$), and thus $\chg b = \chg{(b-1)}+1$, so that  $\chg{(b-1)}  = \be$.
Now, write $\zeta=\ga +a$ with $\ga$ a limit, and as above we see that $a $ is either zero or else $\chg a = \chg{(a-1)} + 1$.
If $a=0$, we have that
\begin{align*}
\chg{\fs\xi n} & = \chg{ \big ( \om^\ga (b-1) +\om^{\fs \ga n}  \big )} = \Om^{\chg \ga} \chg (b-1) +\Om^{\chg{\fs \ga n}} \\
& \stackrel{\text{\sc ih}} =  \Om^{\chg \ga} \be +\Om^{{\fs {\chg \ga} {\chg n}}} = \fs{ \Om^{\chg \ga} (\be+1) }{\chg n}.
\end{align*}
If $\chg a = \al + 1$, we have that
\begin{align*}
\chg{\fs\xi n} & = \chg{ \big ( \om^ \zeta  (b-1) +\om^ \al  n  \big )} = \Om^{\chg \ga + \chg a} \be + \Om^{\chg{(a-1)}}\chg n \\
& = \om^{\chg \ga + \al+1} \be +\om^{\al}\chg n = \fs{ \Om^{\chg \ga + \al +1} (\be+1) }{\chg n}_\Om =  \fs{ \chg \xi }{\chg n}_\Om.&\qedhere
\end{align*}
\end{proof}

\section{Maximality of Base Change}\label{secMax}

Our strategy for proving that every fast Goodstein walk is finite proceeds by showing that the normal forms we have given provide the maximal value after base change, thus yield the longest possible Goodstein processes.
In this section we prove this maximality property.
We begin with some useful lemmas.

\begin{lemma}\label{lemmCZero}
Assume monotonicity of base change below $m$.
Suppose that $m\gknf \A(\xi)$ with $\xi=\al+b$, where $\alpha $ is a limit and $b>0$.
Then, 
\begin{enumerate}

\item $\chg {\A(\al-1+b)} < 1+ \B(\chg \al -1+ \chg b)$, and

\item if $\A(\al-1+b)$ is not in normal form, then $\chg {\A(\al-1+b)}  < 1+ \B(\upsquiggly \al -1+\chg b)$.

\end{enumerate}
\end{lemma}

\begin{proof}
We consider two cases.

\begincases
\item ($\A(\xi-1) $ is in normal form).
Then, monotonicity below $m$ yields $\chg{(-1+b)} \leq -1+\chg b$, hence
\[\chg{\A(\xi-1)} = \B(\chg{(\xi-1)}) \leq \B(\chg\al-1+\chg b).\]
If $\chg b$ is infinite, then the inequality is strict, since $\chg(-1+b)<\chg b = -1+\chg b$.

\item ($\A(\xi-1) $ is not in normal form).
By Lemma \ref{lemmNFBminus}, we have that $\mc\al<b$ and $b \gknf \A(\zeta)$ for some $\zeta>\xi$, which in particular implies that $\zeta>\om$.
Note that in cases where $\A(\xi-1)$ is not in normal form, it suffices to show that $\chg {\A(\xi-1)}  \leq \B(\upsquiggly \al -1+\chg b)$, as $\B(\upsquiggly \al -1+\chg b)\leq \B(\chg \al -1+\chg b)$ by Lemma \ref{lemmSquigProp}.\ref{itSquigVs}.

For the proof to work, we need to show the more general claim that for all $t \in[1,b+1] $,
\[\chg {\A(\al-t+b)}  \leq \B(\upsquiggly \al -t+\chg b).\]
Consider two sub-cases and proceed by induction on $b-t$.

\beginsubcases

\item ($\A(\xi-t)\leq b$).
If $\chg b$ is infinite then $-t+\chg b = \chg b$, and
\[\chg{\A(\xi-t)} \leq \chg b < \B(\upsquiggly \al+\chg b)=\B(\upsquiggly \al-t+\chg b) .\]
If $\chg b$ is finite, $\chg b>2b+2$ by Lemma \ref{lemmBchBigger}, and we have that
\begin{align*}
\B(\upsquiggly \al-t +\chg b) & \geq \B(\upsquiggly\al-b-1+\chg b  ) \geq \B(\upsquiggly \al+ \lceil \nicefrac {\chg b}2\rceil ) \\
& > 2(\nicefrac {\chg b}2) = \chg b \geq \chg{\A(\xi-t)}
\end{align*}
(note that here $\lceil \nicefrac {\chg b}2\rceil$ is the standard integer ceiling function).

\item ($\A(\xi-t) > b$).
By the induction hypothesis,
\[\chg{\A(\xi-t-1)} < 1 + \B(\upsquiggly \al-t-1-\chg b).\]
We claim moreover that $\A(\xi-t-1) \geq b$.
Note that $\mc\al<b$, so that $t<b+1$ (otherwise $\A(\xi-t)=\mc\al$).
Since $\zeta>\xi$, by Lemma \ref{lemmBetweenSlow},
\[ b= \A(\zeta) \notin \big( \A ( \xi - t - 1 ),\A ( \xi - t  ) \big),\]
from which it follows that $b\leq \A ( \xi - t - 1 )$.
Note moreover that if $\la$ is infinite, $b\gknf\A(\zeta)$ with $\zeta>\xi$ implies that $\chg b$ is infinite.

We proceed by induction on $i$ to show that, for $i\leq k$, $ \A^{(i)}(\xi-t) < 1 + \B^{(i)}(\upsquiggly \al  -t +\chg b)$.
By the induction hypothesis for $b-t-1<b-t$, we have that
\begin{align*}
\chg{\A^{(0)}(\xi-t)} & = \chg{\A(\xi-t-1)}  <1+ \B(\upsquiggly \al-t-1-\chg b)\\
& \leq 1+ \B^{(0)}(  \upsquiggly \al-t -\chg b ) .
\end{align*}
Since $\A^{(i)}(\xi-t) \geq \A(\xi-t-1) \geq b$, Lemma \ref{lemmAiNorm} yields $\A^{(i+1)}(\xi-t)  \gknf \A(\fs{\al}{ \A^{(i)}(\xi-t) })$. Hence,
\[
\chg{\A^{(i+1)}(\xi-t)  }  = \B( \chg{\fs{\al}{ \A^{(i)}(\xi-t) }}) = \B(\fs{\upsquiggly \al}{\chg { \A^{(i)}(\xi-t) }}).\]
If $\la$ is finite, then
\[\B(\fs{\upsquiggly \al}{\chg { \A^{(i)}(\xi-t) }})\stackrel{\text{\sc ih}}\leq \B \big (\fs{\upsquiggly \al}{ \B^{(i)}(\upsquiggly \xi -t )} \big ) = \B^{(i+1)}(\upsquiggly \xi -t ) .\]
If $\la$ is infinite, we recall that we defined $\B^{(x)} = \B$ for all $x$.
Then, $ \chg { \A^{(i)}(\xi-t) } <  \B^{(i)}(\upsquiggly \al  -t +\chg b)$ yields
\[ \B(\fs{\upsquiggly \al}{\chg { \A^{(i)}(\xi-t) }}) <\B(\upsquiggly \xi ) = \B^{(i+1)}(\upsquiggly \al -t  +\chg b).\]

The claim follows by setting $i=k$, since
\[\chg{ \A^{(k)}(\xi-t) } <1+ \B^{(k)}(\upsquiggly \al -t+\chg b) \leq  1+  \B(\upsquiggly \al -t+\chg b),\]
as desired.
\qedhere
\end{enumerate}
\end{enumerate}
\end{proof}

\begin{lemma}\label{lemNewLeft}
Let $2\leq k<\la\leq \om$ and write $\uparrow$ for $\uparrow^\la_k$, $\upsquiggly$ for $\upsquiggly^\la_k$.
Suppose that $\A(\xi)$ is in normal form, where $\xi=\al+b$ with $\al$ a limit, and let $i < k$. Then,
\begin{enumerate}

\item $\chg{\A^{(i)}(\xi)} <1+ \B^{(i)}(\chg \xi)$, and

\item\label{itNewLeftTwo} $\B(\chg{\fs\al{\A^{(i)}(\xi)}} ) <1+ \B^{(i+1)}(\chg \xi)$.

\end{enumerate}
\end{lemma}

\begin{proof}
We prove both claims simultaneously by induction on $i$.
When $i=0$, the first claim follows from Lemma \ref{lemmCZero}, or by definition when $\la<\om$ and $\xi$ is a limit (in which case this becomes $\chg{\mc\xi} < 1+\chg{\mc\xi}$).
For $i+1$, the induction hypothesis on the second claim and Lemma \ref{lemmAiNorm} yield
\[\chg{\A^{(i+1)}(\xi)} =\B(\chg{\fs\al{\A^{(i)}(\xi)}} ) <1+ \B^{(i+1)}(\chg \xi) .\]

For the second, Lemma \ref{lemmOrdBch} yields $\B(\chg{\fs\al{\A^{(i)}(\xi)}}) \leq \B({\fs{\chg\al}{\chg{\A^{(i)}(\xi)}}})$.
If $\la$ is finite, then using Lemma \ref{lemmMajorize} we see that
\[  \B({\fs{\chg\al}{\chg{\A^{(i)}(\xi)}}})\stackrel{\text{\sc ih}} \leq \B({\fs{\chg\al}{{\B^{(i)}(\chg\xi)}}}) = \B^{(i+1)}(\chg \xi).\]
If $\la$ is infinite, from the first claim we see that $\chg{\A^{(i)}(\xi)} <  1+\B^{(i)}(\chg \xi)= \B(\chg \xi) $, therefore
\[   \B({\fs{\chg\al}{\chg{\A^{(i)}(\xi)}}}) < \B(\chg \xi)  = 1+ \B^{(i+1)}(\chg \xi). \qedhere \]
\end{proof}

Recall that $\predek\cdot k $ was defined in Definition \ref{defPredek}.
We moreover use $\predek\cdot k^j$ to denote its $j$-fold iteration in the standard way, i.e.~$\predek\xi k^0 = \xi$ and $\predek\xi k^{j+1} = \predek{\predek\xi k^{j }} k $.

\begin{corollary}\label{corMminus}
Assume monotonicity of base change below $m\gknf \A(\xi)$, and let $j_*$ be the unique integer such that
\[0<\lfloor\xi\rfloor^{j_*}_k<\om.\]
Then for all $i\leq k$ and $j< j_*$,
\[\chg{(m-1)} < \B(\chg{\lfloor\xi\rfloor^j_k }) < 1+ \B^{(k)}(\chg\xi) \leq \chg m.\]
\end{corollary}

\begin{proof}
Lemma \ref{lemNewLeft}.\ref{itNewLeftTwo} and induction on $j<j_*$ yield
\[\B(\chg{\lfloor\xi\rfloor^{j+1}_k }) \leq \B(\chg{\lfloor\xi\rfloor^j_k }) < 1+  \B^{(k)}(\chg\xi).\]
Write $\lfloor\xi\rfloor^{j_*-1}_k=\ga+c$ with $\ga$ a limit.
Since $0<\lfloor\xi\rfloor^{j_*}_k<\om$ and $\lfloor\xi\rfloor^{j_*}_k = \fs{\ga}{q}$ for some $q$, by Corollary \ref{corFSInj}, we must have that $\ga=\om$ and thus $\lfloor\xi\rfloor^{j_*-1}_k=\om+c$.
For all $i $,
\begin{align*}
\A^{(i+1)}( {\lfloor\xi\rfloor^{j_*-1}_k }) & =\A (\fs{\om}{\A^{(i)}( {\lfloor\xi\rfloor^{j_*-1}_k })}) \\
& = \A ( {\A^{(i)}( {\lfloor\xi\rfloor^{j_*-1}_k })}) = {\A^{(i)}( {\lfloor\xi\rfloor^{j_*-1}_k })} + 1, 
\end{align*}
and since $m= \A^{(k)}( {\lfloor\xi\rfloor^{j_*-1}_k })$, we see that
$m-1 = \A^{(k-1)}( {\lfloor\xi\rfloor^{j_*-1}_k })$.
Lemma \ref{lemNewLeft} then yields
\begin{align*}
\chg{(m-1)}  & =\chg  \A^{(k-1)}( {\lfloor\xi\rfloor^{j_*-1}_k })\\
& < 1+ \B^{(k-1)}( \chg {\lfloor\xi\rfloor^{j_*-1}_k }) \leq 1+ \B ( \chg {\lfloor\xi\rfloor^{j_*-1}_k }). & \qedhere
\end{align*}
\end{proof}

We are now ready to prove that our normal forms are base-change maximal.

\begin{theorem}\label{theoMax}
If $2\leq k<\la\leq \om$ and $\A(\zeta)\gknf \A(\xi)$, then $\A(\bch {\zeta } {\bases k\la}) \leq\A( \bch {\xi } {\bases k\la}) $.
\end{theorem}

\begin{proof}
Write $\uparrow $ for $\bch{}{\bases k\la}$.
Let $m=\A(\xi)$.
Note that $\zeta<\penum \xi$.
This is because $m\in \big (\A(\penum \xi), \A(\penum \xi+1)\big)$, but by Lemma \ref{lemmOneMore}, if $\zeta\geq \penum \xi$ then $\A(\zeta) \not \in \big (\A(\penum \xi), \A(\penum \xi+1)\big)$.
Thus we write $\xi=  \al+ a$ and $\zeta= \beta+b$ with $\al,\be$ limits and consider two cases.

\begincases
\item $(\zeta<\xi)$.
Let $i_*$ be the unique number such that $0<\lfloor\xi\rfloor_k^{i_*}<\om$.
We claim that there is $i\leq i_*$ such that $\zeta= \lfloor \xi\rfloor^i_k$.
If not, observe that $\A(\zeta) = m$ implies that $\zeta\geq \lfloor\xi\rfloor_k^{i_*}$, for otherwise we would have $\A(\zeta) < \A(\lfloor\xi\rfloor_k^{i_*}) = \A(\xi) $ (since $\A$ is the successor on the natural numbers).
Thus there is $i\leq i_*$ such that $\lfloor \xi\rfloor^{i+1}_k < \zeta \leq \lfloor \xi\rfloor^{i}_k$.
Write $\lfloor \xi\rfloor^{i}_k = \gamma+c$ with $\gamma$ a limit or zero.
If $\zeta<\gamma$, then Lemma \ref{lemmBetweenSlow} yields $\A(\zeta)>\A(\lfloor \xi\rfloor^{i}_k )=\A(\xi)$, contrary to assumption.
Thus $\zeta=\gamma+c'$, but then $\A(\zeta)=\A(\xi)$ yields $c=c'$ and $\zeta= \lfloor \xi\rfloor^{i}_k$.
But by Corollary \ref{corMminus}, $ {\A(\chg{\lfloor \xi\rfloor^{i}_k})}  < \A(\chg\xi)$, as needed.

\item $\zeta\in (\xi,\penum \xi)$.
Using the induction hypothesis (which yields monotonicity below $m$) and Lemma \ref{lemmBchMonOrd}, we see that $ \chg {\zeta} < \chg {\penum \xi} $.
Maximality of $\xi$ yields $\mc{\zeta} \leq \penum m-1 $, so Corollary \ref{corMminus} yields $\mc{\chg {\beta} } < \B^{(k)} (\chg{\penum \xi})$.
Lemma \ref{lemmMajorize} then yields $\B(\chg\zeta) <\B(\chg{\penum \xi}) < \B(\chg\xi)$.\qedhere
\end{enumerate}
\end{proof}



Monotonicity of base change is then immediate from Proposition \ref{propMaxToMon}.

\begin{corollary}\label{corMon}
If $n<m$ and $2\leq k<\la\leq \om$, then $\bch n{\bases k\la} < \bch m{\bases k\la}$.
\end{corollary}

\section{Normal form preservation}\label{secNorm}

The next step in showing that the Goodstein process for our normal forms terminates is to show that normal forms are preserved after base-change, which we will use later to show that the ordinal assignment gives a decreasing sequence of ordinals.

\begin{lemma}\label{lemmBchPreserve}
Let $2\leq k<\ell < \om$ and write $\chg{}$ for $\bch{}{\bases k\ell}$.
If $m\gknf \A_k(\xi)$ and $2\leq k<\ell<\om$, then $\chg m \nfp \ell \A_\ell(\chg \xi)$.
\end{lemma}

\begin{proof}
Write $\A$ for $\A_k$ and $\B$ for $\A_\ell$, and proceed by induction on $m$.
By induction hypothesis we have that $\chg{\penum m} =_\ell \B(\chg{\penum \xi}) $.
By monotonicity (Corollary \ref{corMon}), $\mc{\chg\xi}>\mc{\chg {\penum \xi}}$.
So, in view of Lemma \ref{lemmGeneralNormalForm}, it suffices to show that $\chg m < \B(\chg {\penum \xi} + 1)$ and that if $\theta$ is such that $\B(\theta) \leq \chg m$ and $\mc{\theta}\geq \mc{\chg {\penum \xi}}$, then $\theta\leq\chg \xi$.

To see that $ \B(\chg \xi  ) <  \B(\chg {\penum \xi} +1)$, write $\predek {{\penum\xi} +1} k^i=\gamma_i+c_i $ with $\gamma_i$ a limit.
Then, $\xi  = \gamma_i+d$ for some $i$ and $d<c_i$, since otherwise we would have $\A(\xi ) \geq \A({\penum\xi} +1 )$.
From this it is readily checked that $\B(\chg \xi  ) < \B(\chg {\predek {{\penum\xi} +1} k^i}) < \B(\chg {\penum\xi} +1)$.

With this, it remains to show that $\chg{\xi }$ is maximal.
Towards a contradiction, assume that $\hat \theta $ is minimal with the property that $\hat\theta >\xi $, $\mc{\hat \theta }\geq \chg{\penum m}$, and $\B(\hat \theta )\leq \chg m$.
Since $\B(\hat\theta ) = \B(\lfloor{\hat \theta }\rfloor_\ell)$, we must have that $\lfloor{\hat \theta }\rfloor_\ell\leq \chg {\xi}$.
Write $\xi= \al+a$ and $\hat\theta = \hat\be+\hat b$ with $\al,\hat \be$ limits, so that $\chg\xi=\chg\al+\chg a$.
We cannot have $\chg \al=\hat\be$, since in this case $\chg b>\hat a$, and $\B(\chg{\xi})<\B({\hat\theta})$.
Thus $\hat\be>\chg \al$.
By minimality of $\hat\theta $, either $\mc{\hat \be}\geq \chg {\penum m}$ and $\hat b  = 0$, or else $\mc{\be}< \chg {\penum m}$ and $\hat b  = \chg{\penum m}$.
In either case, there is $b$ such that $\chg b= \hat b $.

By Lemma \ref{lemTrunc}, there is a truncation $\hat\zeta $ of $\chg\xi$ and some $\hat c\geq \B^{(k)} (\hat\theta)$ such that $\hat\zeta = \fs{\hat\be}{\hat c}  $.
Since $\hat\zeta$ is a truncation of $\chg\xi$, we have that $\hat\zeta=\chg\zeta$ for some $\zeta<\xi$; in particular, $\hat c= \chg c$ for some $c$.

Let $\theta= \lceil\zeta\rceil +b  $.
Lemma \ref{lemmCeil} imples that $ {\fs {\lceil\zeta\rceil } c} =  \zeta $ and $\hat\be = \upsquiggly \be $.
We aim to prove that $\A(\theta)\leq m$, and proceed by assuming $\A(\theta)> m$ toward a contradiction.
We claim that for all $i<k$,
\begin{enumerate}[label=(\alph*)]

\item\label{claimPreserveOne} $\chg{  \A^{(i)} (\theta) }   \leq   \B^{(i)} (\hat\theta) $,

\item\label{claimPreserveTwo} $  \A^{(i)} (\theta)  < c$,

\item\label{claimPreserveThree} $  \A^{(i+1)} (\theta)  < m$, and

\item\label{claimPreserveFour} $\A (\fs\theta{\A^{(j)} (\theta) })$ is in normal form for $j<k-1$.

\end{enumerate}
The third item for $i=k-1$ will yield our desired claim, since it becomes $\A(\theta)<m$, contradicting that $\A(\xi)$ is in normal form.
However, we must prove the four claims simultaneously for induction on $i$ to go through.

Assume that all four claims hold for $i'<i$.
Let $e = A^{(i)}(\theta)$.

\proof[Proof of \ref{claimPreserveOne}.] When $i=0$, we have by Lemma \ref{lemmBchBigger} that
\[\chg e = \chg{\A(\theta-1)} \leq \B (\hat\theta-1) = \B^{(0)}(\hat\theta)  ,\]
establishing the first claim when $i=0$.
For $i>0$, we use the induction hypothesis on the fourth claim to see that
\begin{align}
\nonumber\chg e &= \chg{ \A ( \fs \be{\A^{(i-1)} (\theta)})} \\
\nonumber& = \B( \chg{\fs \be{\A^{(i-1)} (\theta)}} ) = \B(\fs{\upsquiggly \be}{\chg{ \A^{(i-1)} (\theta)}}) \\
\label{itApplyIH}& \leq \B(\fs{\upsquiggly \be}{{ \B^{(i-1)} (\hat \theta)}}) = \B^{(i)} (\hat\theta) ,
\end{align}
where in \eqref{itApplyIH} we have used the induction hypothesis to see that $ \chg{ \A^{(i-1)} (\theta)} \leq \B^{(i-1)} (\hat\theta) $.
This establishes the first claim when $i>0$.

\proof[Proof of \ref{claimPreserveTwo}.]
Using the previous claim we see that
\[\chg e = \chg \A^{(i)} ( \theta) \leq \B^{(i)} (\hat\theta)  < \B^{(k)} (\hat\theta) <\chg c,\]
and by monotonicity of the base change operator, $e < c$, establishing the second claim.

\proof[Proof of \ref{claimPreserveThree}.]
Since $e<c$, we have that $\fs \be e <\xi$ and $\mc{\fs\be e}\leq \mc\zeta \leq \mc\xi$.
By Lemma \ref{lemmMajorize}, $\A^{(i+1)}(\theta) = \A(\fs\be{e}) < \A(\xi) = m$.

\proof[Proof of \ref{claimPreserveFour}.]
Let $j<k-1$ and $\chi = \fs \be e$.
Note that the previous item yields $\A(\chi)<m<\A(\penum \xi+1)$.
Since $\mc\chi\geq\penum m$, by Lemma \ref{lemmGeneralNormalForm}, it suffices to show that if $\mc\eta \geq \penum m$ and $\eta>\chi$, then $\A(\eta)>\A(\chi)$.
If $\eta>\xi$, then the maximality of $\xi$ yields $\A(\eta)>\A(\xi)>\A(\chi)$.
If $\eta\leq \xi$, then $\eta\in(\chi,\theta )$, which by Lemma \ref{lemmBetweenSlow}, implies that $\A(\eta)>\A(\chi)$.
Thus $\A(\chi)$ is in normal form, as required.
\medskip

Applying \ref{claimPreserveThree} with $i=n-1$, we conclude that $\xi$ is not maximal with $\mc\xi\geq \penum m$ and $\A(\xi)\leq m$, contradicting the original assumption that $\A(\xi)$ is in normal form.
Thus we conclude that $\B(\chg\xi)$ is in normal form as well.
\end{proof}

With this and a simple induction, we obtain the following useful property.

\begin{corollary}\label{corPreserve}
If $2\leq k<\ell <\la \leq \om$ and $m\in\mathbb N$, then
\[\bch m{\bases k\la} = \bch {\bch m{\bases k\ell}}{\bases \ell\la}.\]
\end{corollary}

\section{Fast Goodstein Walks}\label{secWalk}

Now we are ready to define our fast Goodstein processes and prove that they terminate.
Using base-change maximality, we will also show that Goodstein walks based on the $\A$ function always terminate, even if normal forms are not used.

\begin{definition}
Given a natural number $m$, we define a sequence $\big (\good mi \big )_{i< \alpha}$, where $\alpha \leq \omega$, by the following recursion.
\begin{enumerate}

\item $\good m0 = m$;

\item if $\good mi > 0$, then $\good m{i+1}  = \bch {\good mi} {\bases{k+2}{k+3}} - 1$;

\item if $\good mi = 0$, then $\alpha = i+1$ and the sequence terminates.

\end{enumerate}
The sequence $\big (\good mi\big )_{i<\al}$ is the {\em Fast Goodstein sequence starting on $m$.} 
\end{definition}

\begin{theorem}\label{theoGood}
Given any $m\in \N$, the Fast Goodstein sequence starting on $m$ terminates on finite time.
\end{theorem}

\proof
Let $\big (\good mi \big )_{i<\al}$ be the fast Goodstein sequence starting on $m$. Let $i< \al$. Then,
\begin{align}
\nonumber  {\uparrow_{i+3}^\om }\good m{i+1}   & = {\uparrow_{i+3}^\om} ( \bch {\good mi }{\bases {i+2}{i+3}} - 1) \\
\label{itTermOne}&< {\uparrow_{i+3}^\om } \bch {\good mi}{\bases {i+2}{i+3}}  \\
\label{itTermTwo}& = {\uparrow_{i+2}^\om} \good mi ,
\end{align}
where \eqref{itTermOne} follows from Corollary \ref{corMon} and \eqref{itTermTwo} from Corollary \ref{corPreserve}.
Hence $\big (\uparrow_{i+2}^\om \good mi \big )_{i < \al}$ is a decreasing sequence of ordinals, so $\al$ must be finite.
\endproof

It is not needed to write numbers in normal form in order for the process to terminate.
Natural numbers may be represented using the $\A_k$ functions in various ways.
To make this precise, we build terms for numbers and ordinals out if this function.
Given fixed $k\geq 2$, the set of {\em $k$-terms} and {\em ordinal $k$-terms} are defined inductively as follows:
\begin{enumerate}

\item $0$ is both a $k$-term and an ordinal $k$-term.

\item If $t$ is an ordinal term, then $\A_k(t)$ is both a $k$-term and an ordinal $k$-term.

\item If $t,{s}$ are ordinal terms and $r$ is a number term, $\om^tr+{s}$ is an ordinal term.

\end{enumerate}

We remark that we use the same notation for function symbols and the functions they represent, but whether an expression should be treated as a term or as a number will always be made clear.
The set of $k$-terms will be denoted $\mathbb T_k$, and we set $\mathbb T=\bigcup_{k<\om}\mathbb T_k$.
The {\em value} of a term is defined inductively in the obvious way by $\val 0 = 0$, $\val{\A_k(t)} = \A_k(\val t)$, and $\val{\om^tr+{s}} = \om^{\val t}\val r+\val{s}$ (here, the left-hand side of the equality should be regarded as a term, the right hand as an ordinal).


The base change operator can be applied to arbitrary terms, even those not in normal form.
Given $k\leq \ell$ and $t \in \mathbb T _k$, we define $\bch t{ ^\ell} \in \mathbb T _\ell $ recursively by
\begin{enumerate}
\item $\bch 0{^\ell} = 0$,

\item $\bch {\A_k(t)}{^\ell} = \A_\ell(\bch t{^\ell})$,

\item  $\bch {(\om^tr+{s})}{^\ell} = \om^{\bch t{^\ell}}{\bch r{^\ell}}+{\bch {s}{^\ell}}$.
\end{enumerate}

Then, normal forms give maximal base change in the following sense.

\begin{proposition}\label{propAMax}
If $2\leq k<\ell<\om$ and $t$ is any $k$-term, then $\bch t {^\ell} \leq \bch{\val t}{\bases k\ell}$.
\end{proposition}

\begin{proof}
Follows from Theorem \ref{theoMax} using induction on term complexity. 
\end{proof}

With this, we may define Goodstein walks, in which natural numbers are written using {\em any} term.

\begin{definition}
A {\em fast Goodstein walk} is a sequence $(m_i)_{i=0}^\alpha$, where $\alpha \leq \omega$, such that for every $i<\alpha$, there is an $(i+2)$-term $t _i$ with $\val{t_i} = m_i$ and $m_{i+1} = \bch {t_i} {^{i+3}} - 1$.
\end{definition}

\begin{theorem}\label{theoWalk}
Every fast Goodstein walk is finite.
\end{theorem}

\proof
Let $(m_i)_{i=0}^\alpha$ be a Goodstein walk for $\mathcal F$.
Let $m = m_0$.
By induction on $i$, we check that $m_i \leq \good mi$.
For the base case this is clear. Otherwise, $m_{i+1} = \val{\bch{t_i}{^{i+3}}} - 1$ for some $(i+2)$-term $t_i$, and thus
\[m_{i+1} = \val{\bch{t_i}{^{i+3}}} - 1 \leq \bch{m_i}{\bases{i+2}{i+3}} - 1 \stackrel{\text{\sc ih}}\leq \bch{\good m{i}}{\bases{i+2}{i+3}} - 1 = \good m{i+1} ,\]
where the second inequality uses Corollary \ref{corMon} along with the induction hypothesis for $i$.
Thus if we choose $i$ such that $\good mi =0 $, we must have $\alpha \leq i $.
\endproof

\begin{example}\label{exAlt}
Consider alternative normal forms obtained by writing $m \simeq_k \A_k (\xi)$, where $\xi$ is maximal so that $\A(\xi) = m $.
Such normal forms give alternative Goodstein sequences, which are terminating by Theorem \ref{theoWalk}.
\end{example}

\section{Fundamental sequences for the Bach\-mann-Ho\-ward ordinal}\label{secThetaFun}

The strategy for showing that Theorem \ref{theoGood} is not provable in $\sf KP$ is to compare it to the process of descending along fundamental sequences for $\B(\ve_{\Om+1})$, whose termination is already known not to be provable \cite{Eguchi,FWTheta}
In the remainder of this section, we write $\B=\A_\om$.
The fundamental sequences we will use are based on those defined by Buchholz for the $\vartheta$ function~\cite{BuchholzOrd}.
We have shown that very similar fundamental sequences can be defined for $\B$ (also denoted $\sigma$), and that the two functions coincide for $\xi\geq \Om^2\cdot\om$~\cite{FWTheta}.
In particular, we have that $\vartheta(\ve_{\Om+1}) = \B(\ve_{\Om+1} )$, the {\em Bachmann-Howard ordinal.}\footnote{Note that $\ve_{\Om+1}$ is not officially in the domain of $\B$, but we may define $  \B(\ve_{\Om+1} ) = \lim_{n\to\infty}  \B(\Om_n ) $.}

We need some auxiliary definitions before giving the fundamental se\-quen\-ces for $\B$.
We will use the function $\tau_\Om(\xi)$ defined in Section \ref{secOrd}, and write $\tau$ for $\tau_\Om$.

The fundamental sequences for the $\B$ function require a case distinction depending on whether the value of $\B(\xi)$ has a `jump' at $\xi$; it could be either that $\xi$ is a limit, or that $\xi$ is a limit but it is not the case that $\B(\fs\xi{\tau_n}) \to \B(\xi) $ as $\tau_n\to \tau$.
This occurs when the following holds (see \cite{BuchholzOrd,FWTheta}).

\begin{definition}
We define sets
\begin{enumerate}
\item 
${\rm FIX}  = \{\xi<\ve_{\Om+1}: \fs\xi 1^*<\xi^* =\tau(\xi) = \B_\mathbb X(\ga)\text{ for some $\ga>\xi$}\}
$, and

\item 
${\rm JUMP}  = \{0\}\cup {\rm Succ} \cup {\rm FIX} $.

\end{enumerate}
\end{definition}

In order to ensure that the fundamental sequences converge in such cases, we need to define an auxiliary value, essentially equivalent to Buchholz's $\vartheta^*$.

\begin{definition}
For $\xi<\ve_{\Om+1}$, we set
\[
\B^{\{0\}} (\xi) =
\begin{cases}
\B(\zeta)&\text{if $\xi =\zeta+1$}\\
 \tau(\xi)&\text{if $\xi\in {   \rm FIX} $}\\
 0&\text{otherwise.}\\

\end{cases} 
 \]
If $\xi=\al+\be$ with $\al=\Om\tilde\al$, we define $\B^{\{i+1\}}(\xi)$ recursively by
\[\B^{\{i+1\}}(\xi) = \B\big ( \fs\al{\B^{\{i\}}(\xi)} \big ),\]
and we set
\[
\check \xi=
\begin{cases}
 \al&\text{if $\B^{\{0\}}(\xi) >0$}\\
\xi&\text{otherwise.}
\end{cases}
\]
\end{definition}

With this, we may now define the fundamental sequences we will use.
It will be convenient to define fundamental sequences for some uncountable ordinals, and thus the domain of our fundamental sequences will be $\Lambda\times\mathbb N$, where $\B(\ve_{\Om+1}) \subsetneq\Lambda\subsetneq \ve_{\Om+1}$ is as specified below.

\begin{definition}
Let $\xi<\B(\ve_{\Om+1})$, $n<\om$, and define
\[\Lambda = \{\xi <\ve_{\Om+1} : \mck\Om\xi<\B(\ve_{\Om+1}) \text{ and } \tau(\xi)<\Om\} .\]
We define $\cfs\cdot\cdot\colon \Lambda \times\mathbb N \to \B(\ve_{\Om+1})$ by:
\begin{enumerate}

\item $\cfs 0n = 0$ for all $n$.

\item If $\be<\Om$, then 
$\cfs{\B( \be) } n  = \be$.

\item

If $\xi = \fs\al\tau \geq \Om$, then $\cfs\xi n = \fs\al{\cfs \tau n}$.

\item If $0<\tau(\check \xi)<\Om$, then
$\cfs{\B(\xi)}n  = \B \big ( \cfs{\check \xi} n +\B^{\{0\}}(\xi) \big )$.

\item If $\tau(\check \xi)=\Om$, then
$
\cfs{\B(\xi)}n=
\B^ {\{n \}} (\xi).
$

\end{enumerate}
\end{definition}

Recall from the introduction that $\sf KP$ is a restriction of $\sf ZFC$ with proof-theoretic ordinal $\B(\ve_{\Om+1} )$.
This ordinal can be used to bound the provably total computable functions of $\sf KP$.

\begin{definition}\label{defFFun}
For $i<\om$ and $\al<\B(\ve_{\Om+1})$, define $\fsi\al i$ recursively by $\fsi\al 0 = \al $ and $\fsi\al {i+1} = \cfs {\fsi \al i}{i+1} $.
Define $F(n)$ to be the least $\ell$ such that
$\fsi{\B(\Om_n)} \ell = 0$.
\end{definition}

The function $F$ is total since $ \fsi{\B(\Om_n)} {i+1} < \fsi{\B(\Om_n)} i$ whenever the right-hand side is not zero, but totality is not provable in $\sf KP$.
In fact the following, more general, claim holds; it is a special case of a general principle of Cichon et al.~\cite{BCW} adapted to our system of fundamental sequences~\cite{FWTheta}.

\begin{theorem}\label{theoKPInc}
Let $\varphi$ be $\Sigma_1^0$ formula and suppose that ${\sf KP}\vdash \forall x\exists y \varphi(x,y) $.
Let $f_\varphi(n)$ be the least $m$ such that $\varphi(n,m)$.
Then, $\exists m \ \forall n>m \ \big ( f_\varphi(n) < F(n) \big)$.
\end{theorem}

In order to compare $F(n)$ to the length of our fast Goodstein walks, the following property will be useful.
It holds in general for any system of fundamental sequences with the Bach\-mann property \cite{FernandezWCiE}.

\begin{proposition}\label{propMajorize}
Let $(\xi_n)_{n\in \mathbb N}$ be a sequence of ordinals below $\B(\ve_{\Om+1})$ such that, for all $n$, $\cfs{\xi_n}{n+1} \leq \xi_{n+1} \leq \xi_n$.
Then, for all $n$, $\xi_n \geq \fsi\xi n$.
\end{proposition}

Proposition \ref{propMajorize} thus allows us to compare sequences of ordinals globally by considering only their local behavior.

\section{Independence}\label{secInd}

Our strategy to prove that Theorems \ref{theoGood} and \ref{theoWalk} are not provable in $\sf KP$ is to show that they grow at least as fast as the function $F$ of Definition \ref{defFFun}.
By Proposition \ref{propMajorize}, it suffices to show that $\cfs{\bch{m}{\bases {k+2}\om}}k \leq \bch{(m-1)}{\bases {k+2}\om}$ for all $k$, as this will allow us to conclude that $\fsi{\bch{m}{\bases { 2}\om}  } k \leq   \bch{\good mk }{\bases{k+2}\om}  $.
By choosing suitable $m = m(n)$, this will show that the termination time for the fast Goodstein process is bounded below by $F(n)$.
We begin with a useful lemma comparing the auxiliary value $\B^{\{0\}} (\chg \xi)$ to $\chg \A( \xi - 1)$.
As in the previous section, we write $\tau$ for $\tau_\Om$.

\begin{lemma}\label{lemmThetaStar}
Let $k\geq 2 $ and write $\uparrow$ for $\uparrow_k^\om$, $\A=\A_k$, and $\B=\A_\om$.
Then, for every $\xi<\ve_0$,
\[\B^{\{0\}} (\chg \xi) \leq \chg \A( \xi - 1).\]
\end{lemma}

\begin{proof}
Consider the following cases.
\begincases

\item ($\chg \xi$ is a successor).
Then, $\chg \xi -1=\chg{(\xi-1)}$.
By Theorem \ref{theoMax}, we see that
$\chg \A(\xi-1)  \geq \B \big (\chg { \xi } - 1 \big ) = \B^{\{0\}} (\chg \xi)$.

\item ($\B^{\{0\}}(\chg\xi) = \tau(\chg\xi)$).
Lemma \ref{lemmBoundTwo} yields $\mc\xi<\A(\xi-1)$, so by Corollary \ref{corMon}, $ \tau(\chg\xi) \leq  \mck\Om{\chg \xi} < \chg \A(\xi-1)$.

\item (other cases).
Then $\B^{\{0\}}(\chg\xi) = 0$, and the claim is trivially true.\qedhere
\end{enumerate}
\end{proof}

We need one more preliminary lemma involving fundamental sequences.
Since we will use this in the proof by induction that $\cfs{\bch{m}{\bases {k+2}\om}}k \leq \bch{(m-1)}{\bases {k+2}\om}$ for all $m$ and $k$, we may assume that this inequality already holds below $\mc\xi$.
Once we have proven Proposition \ref{propMinvsFS} below, this assumption may be dropped.

\begin{lemma}\label{lemmClaim}
Let $k<\om$ and $\xi = \om\tilde\xi  $ with $0<\tilde\xi<\ve_0$.
Write $\chg{}$ for $\bch{ }{\bases {k+2}\om}$, $\A$ for $\A_k$, and let $\tau = \tau(\chg\xi)$, $\theta = \B^{\{0\}}(\chg \xi) $ and $A =   \A(\xi -1)$.

Suppose that for all $m\leq \mc{\xi}$, we have that $\cfs{\chg{m}}k \leq \bch{(m-1)}{\bases {k+2}\om}$.
Then,
\[\fs{\chg\xi} {\cfs \tau k}_\Om +\theta \leq \chg {\fs{\xi}A  }.\]
\end{lemma}

\begin{proof}
Write $\xi=\om^{\al}c +\be$ in $\omega$-normal form and consider the following cases.
\begincases
\item ($\beta >0$). Then we have inductively that

\begin{align*}
\fs{\chg\xi} {\cfs\tau k}_\Om+\theta &= {\Om^{\chg{\al}}\chg c} + \fs{\chg{\be}} {\cfs\tau k}_\Om +\theta  \\
&\leq {\Om^{\chg{\al}}\chg c} + \chg {\fs{\be}A  } = \chg {\fs{\xi}A  }.
\end{align*}

\item ($\be = 0$).
Consider the following sub-cases.
\beginsubcases

\item ($\chg c \in {\rm Lim} $).
Then, $\chg c=\tau$ by the definition of $\tau$, so that the assumption yields $\cfs{\tau}k \leq \bch{(c-1)}{\bases {k+2}\om}$. 
Moreover, $  {\fs{\xi}A  }  =  \om^{\al}(c-1)+ \delta $ for some $\delta$ with $\mc\delta \geq A$.
Note that Lemma \ref{lemmThetaStar} yields $\chg A \geq \theta$, so that also $\chg\delta\geq\theta$.
From this we see that
\begin{align*}
\chg {\fs{\xi}A  } & = \chg{(\om^{\al}(c-1)+ \delta )} = \Om^{\chg{\al}} \chg {(c-1)}+\chg \delta\\
& \geq  \Om^{\chg{\al}} \cfs\tau k+ \theta = \fs{\chg{\xi}} {\cfs\tau k}_\Om +\theta.
\end{align*}

\item ($\chg c \in {\rm Succ} $).
Write $\chg(c-1)=\eta$, so that $\chg c=\eta+1$.
Here we consider further sub-cases according to $\al$.
\beginsubcases
\item ($\al=\om\al'$).
Then, $\tau(\chg \xi) =\tau(\chg \al)$, so the induction hypothesis and Lemma \ref{lemOrdIneq} yield
\begin{align*}
\fs{\chg{\om^{\al}}c} {\cfs\tau k}_\Om+\theta & = \chg{\om^{\al}}(c-1)+ \Om^{\fs{\chg{ {\al}}} {\cfs\tau k}_\Om}+\theta\\
& \leq \chg{\om^{\al}}(c-1)+ \Om^{\fs{\chg{ {\al}}} {\cfs\tau k}_\Om+\theta}\\
&\stackrel{\text{\sc ih}}\leq \Om^{\chg{ { \al}  }}\eta + \Om^{\chg{\fs{ \al} A}} = \chg{\fs{\om^{\al}} A }.
\end{align*}

\item ($\al \in {\rm Succ} $).
Then, $\al = \om\al' + t $ for some finite $t>0$.
The assumption that $\tau(\xi) = \tau$ yields $\chg t=\tau$, and by assumption, $\chg(t-1)\geq \cfs\tau k$.

If $\theta =0$, note that $A\geq 1$, so
\begin{align*}
\chg{\fs{\xi}A} & =\chg{\fs{\om^{\om\al' + t }c} A} = \chg{\big (\om^{\om\al' + t }(c-1) + \om^{\om\al' + t -1 }A \big )}\\
& \geq    \chg{\big ({\om^{\om\al' + t }(c-1) +\om^{\om\al' + t -1 }\big) }}
 = \Om^{\chg {\al}}\eta +   {\Om^{\chg\om\al' + \chg{(t -1)} }  }\\
 &  \stackrel{\text{\sc ih}}\geq \Om^{\chg {\al}}\eta + {\Om^{\chg\om\al' + \cfs\tau k } } = \fs{\chg{\om^{\al}c}}{\cfs\tau k}_\Om +\theta .
\end{align*}

Otherwise, $\theta $ is a limit ordinal, so that $A\geq\theta>1$.
Since also $\om^{\om\al' + t -1}  >1$, using Lemma \ref{lemOrdIneq} we see that 
\begin{align*}
\chg{\fs{\xi}A} & =\chg{\fs{\om^{\om\al' + t }c} A} = \chg{ \big ( \om^{\om\al' + t }(c-1) + \om^{\om\al' + t -1 }A \big ) }\\
& \leq \chg{( \om^{\om\al' + t }(c-1) + \om^{\om\al' + t -1 }+ A)}\\
& =    {\Om^{\chg{\al}}\eta + \Om^{\chg\om\al' + \chg{(t -1)} } + \chg A} \stackrel{\text{\sc ih}}\geq {\Om^{\chg{\al}}\eta + \Om^{\chg\om\al' + \cfs{\tau}k } + \chg A}\\
& = \fs{\chg{\om^{\al}c}}{\cfs\tau k}_\Om+\chg A \geq \fs{\chg{\om^{\al}c}}{\cfs\tau k}_\Om+\theta.&\qedhere
\end{align*}
\end{enumerate}
\end{enumerate}
\end{enumerate}
\end{proof}

With this, we are ready to show that the Goodstein process decreases more slowly in each step than the fundamental sequences do.

\begin{prop}\label{propMinvsFS}
If $m\in \mathbb N$ and $k<\om$, then
\[\cfs{\bch{m}{\bases {k+2}\om}}k \leq \bch{(m-1)}{\bases {k+2}\om}.\]
\end{prop}

\begin{proof}
Proceed by induction on $m$.
Write $\uparrow$ for $\uparrow^\om_{k+2}$, $\A$ for $\A_{k+2}$, $\B$ for $\A_\om$, and let $m\nfp{k+2}\A(\xi)$.
Consider the following cases.
\begincases

\item ($\xi<\om$).
Then $m-1=\xi$ and $\chg m = \B(\chg \xi) = \chg\xi+1$, so
$\chg {(m-1)} = \chg \xi  = \cfs { \chg { m } }k$.

\item ($\om\leq \xi$).
Write $\xi=\al+b$ with $\al$ a limit and consider two sub-cases, according to $\widecheck{\chg\xi}$.
\beginsubcases

\item ($\widecheck{\chg\xi} > \chg\al $).
By the definition of $\check\cdot$, this is only possible if $\B^{\{0\}}(\chg\xi)=0$ and $\chg b>0 $; by the definition of  $\B^{\{0\}}(\chg\xi)$, we must also have $\chg b\in \rm Lim$.
Checking the definitions of the fundamental sequences, we have that $\cfs{\B(\chg\xi)}k  = \B( \chg \al+ \cfs{\chg b}k  ) $, hence
\begin{align*}
\chg {(m-1)} &=\chg { \big (\A(\xi)-1\big ) } \geq  \chg \A(  \xi-1  ) \\
& \geq \B \big ( \chg \al + \chg (b-1)\big )
\stackrel{\text{\sc ih}}\geq  \B( \chg \al+ \cfs{\chg b}k  ) = \cfs{\B(\chg\xi)}k .
\end{align*}

\item ($\widecheck{\chg\xi} = \chg\al $).
Let $\theta = \B^{\{0\}} (\chg \xi)$ and $A =   \A(\xi-1)$, so that by
Lemma \ref{lemmThetaStar}, $\theta \leq \chg A$.
Let $\tau= \tau( \chg\xi  )$ and consider two sub-cases.
\beginsubcases

\item ($0<\tau<\Om$).
By Lemma \ref{lemmClaim}, $\fs{\chg\al} {\cfs \tau k}_\Om +\theta \leq \chg {\fs{\al}A  }$. 
Moreover,
\[\mck\Om{  \fs{\chg\al} {\cfs \tau k}_\Om +\theta} \leq \max\{ \mck\Om{\chg\al}, \theta\} < \chg A \leq \mck \Om{  \chg { \fs{\al}A  }},\]
so that the induction hypothesis and Lemma \ref{lemmClaim} yield
\begin{align*}
\cfs{ \B( \chg\xi) }{k}_\Om & = \B( \fs{\chg\al}{\cfs\tau k}_\Om+\theta)
\leq \B( \chg {\fs{\al }A  } )\\
&\stackrel{\text{\sc ih}}= \B(\chg {\fs{ \al }A  }  )  \leq  \chg{\A (  {\fs{ \al }A  })}  = \chg \A^{(1)} (\xi)   < \chg \A(\xi).
\end{align*}

\item ($\tau  = \Om$).
Define $\theta_i = \B^{\{i\}}( \chg \al +\theta)$.
We claim that $\theta_i\leq \chg{\A^{(i)}(\xi)} $ for all $i$.
Since $\chg{(m-1)}\leq \chg{\A^{(k+1)}(\xi)}$ and $\cfs{\B(  { \chg \al} ) }k =\theta_k$, this yields the desired result.

For $i=0$, we already have that $\theta \leq \chg\A(\xi-1) = \chg{\A^{(0)}(\xi)}$.
For $i+1$, we see that
\begin{align*}
\theta_{i+1} & = \B( \fs{\chg \al }{\theta_i}_\Om)\\
& \leq \B(  \fs{\chg \al }{ \chg{\A^{(i)}(\xi)} }_\Om ) &\text{by induction on $i$}\\
& = \B(  \chg{\fs{  \al }{  {\A^{(i)}(\xi)} }}  ) &\text{by Lemma \ref{lemmUpSame} and $\tau(\chg\al)=\Om$}\\
& = \chg \A(   \fs{  \al }{ {\A^{(i)}(\xi)} }  ) \\
& = \chg \A^{(i+1)}(\xi). & \qedhere
\end{align*}

\end{enumerate}
\end{enumerate}

\end{enumerate}
\end{proof}

The following corollary, while not used explicitly for our main results, shows that our ordinal assignment is surjective.

\begin{corollary}
For all $\xi < \B (\ve_{\Om+1})$, there exist $k\geq 2$ and $\xi<\ve_0$ such that $\zeta= \bch \xi{\bases k\om}$.
\end{corollary}

\begin{proof}
If $\xi < \B (\ve_{\Om+1})$ then $\xi < \B (\Om_i) $ for some $i$, and
\[\bch {\A(\om_i)} {\bases 2\om} \geq \B(\bch{\om_i}{\bases 2\om} ) = \B(\Om_i) >\xi.\]
Thus there exist $m' $ and $k'\geq 2$ such that $\bch {m'}{\bases {k'}\om}\geq \xi$.
Let $\xi'\geq \xi$ be least with the property that $\bch {m'}{\bases {k'}\om}\geq \xi'$ for some $m'$ and $k'\geq 2$.

We claim that $\xi'=\xi$.
If not, let $k\geq k'$ be large enough so that $\cfs{\xi'}k \geq \xi $, and let $m = \bch {m'}{\bases{k'}{k+2}}$.
By Proposition \ref{propMinvsFS}, we have that
\[\xi\leq\cfs{\xi'}k \leq \bch{(m-1)}{\bases {k+2}\om} < \bch{ m }{\bases {k+2}\om} = \xi'. \]
But $\zeta:=\bch{(m-1)}{\bases {k+2}\om} $ contradicts the minimality of $\xi'$.
Thus we conclude that $\xi'=\xi$, as desired.
\end{proof}

\begin{corollary}\label{corPsivsFS}
Given $m,k\in\mathbb N$, $\bch {\good mk}{\bases {k+2}\om} \geq \fsi{\bch m{\bases 2\om}} k$.
\end{corollary}

\begin{proof}
We have that
\begin{align*}
\bch{\good m{k+1}}{\bases {k+3}\om} & =  \bch{\big ( \good mk-1 \big )}{\bases {k+3}\om} \geq \cfs{\good mk}k,
\end{align*}
where the inequality is an instance of Proposition \ref{propMinvsFS}.
Proposition \ref{propMajorize} then yields $\bch {\good mk}{\bases {k+2}\om} \geq \fsi{\bch m{\bases 2\om}} k$ for all $k$.
\end{proof}

\begin{theorem}\label{theoGoodInc}
Theorem \ref{theoGood} (and hence Theorem \ref{theoWalk}) is not provable in $\sf KP$.
\end{theorem}

\proof
Let $m(n):=\A_2 (\om_{n})$.
It is not hard to check that $m(n)$ is in normal form, as $\xi>\om_n$ implies also that $\mc\xi\geq \mc{\om_n}$ and hence $\A(\xi)>\A(\om_n)$ by Lemma \ref{lemmMajorize}. 
Then, $\bch{m(n)}{\bases 2\om } = \bch{\A(\om_{n})}{\bases 2\om } =  \B (\Om_{n})$.
Let $G(n)$ be the least value of $k$ such that $\good{m(n)}k = 0$.
By Corollary \ref{corPsivsFS},
\[ \bch{\good {m(n)}k}{\bases {k+2}\om }  \geq \fsi{\bch{m(n)}{\bases 2\om }}k = \fsi{\B(\Om_{n}) }k .\]
Hence, $G(n) \geq F(n)$, as any value of $k$ with $ \good {m(n)}k =0$ also has $\fsi{\B(\Om_{n}) }k = 0$.
The function $G(n)$ is clearly computable, hence definable by a $\Sigma^0_1$ formula.
It follows that $\forall x \exists y \ \big ( G(n) =y \big )$ is not provable in $\sf KP$, hence neither is Theorem \ref{theoMax}.
\endproof

\section{Concluding remarks}\label{secConc}

We have proven that Goodstein processes based on the $\A$ function always terminate, leading to independence results of strength the Bach\-mann-Howard ordinal.
This opens various lines of research.
The $\A$ function is not the only fast-grow\-ing function based on transfinite recursion below $\ve_0$.
A natural question is how sensitive the termination and independence results presented here are to the precise choice of fast-grow\-ing functions used (e.g.~the aforementioned Hardy function~\cite{Hardy}), and in particular if the normal forms based on successive maximization we used yield maximality of base change in a more general context.

Variation of normal forms also leads to questions regarding independence.
Note that the alternative normal forms of Example \ref{exAlt} give rise to a terminating Goodstein process, but our methods do not establish any lower bounds on such processes, so it is not immediately obvious whether termination is provable in $\sf KP$.
We conjecture that it is not, and remark that independence for this alternative Goodstein principle would also imply Theorem \ref{theoGoodInc}.
Thus a different strategy for the current work would have been to prove that the maximal Goodstein principle terminates, and that the alternative one leads to independence, thus obtaining two independence results at once. 
The drawback of such a `dual' approach is that, at least in this case, the two Goodstein processes would have to be studied separately, for example providing different ordinal assignments for each.
We thus leave the analysis of this alternative Goodstein process for future work, but such a `dual' approach may be interesting when studying less involved Goodstein processes.

Another direction involves Goodstein processes based on fast-grow\-ing hierarchies up to ordinals $\Lambda>\ve_0$, perhaps based on the Bach\-mann-Howard ordinal, or even lager ordinals where suitable systems of fundamental sequences are known, e.g.~the ordinal of $\Pi^1_1$-${\sf CA}_0$.
The challenge is in identifying the base-change maximal normal forms in these contexts, and the proof-theoretic strength of termination.

\bibliographystyle{plain}
\bibliography{biblio}

\end{document}